\documentclass{amsart}                 

\allowdisplaybreaks                     
\usepackage{amssymb,amsmath, amsfonts}
\usepackage{hyperref}
\usepackage[usenames]{color}
\usepackage{latexsym}

\theoremstyle{plain}
\newtheorem{theorem}{Theorem}[section]
\newtheorem*{theorem*}{Theorem}
\newtheorem{lemma}[theorem]{Lemma}
\newtheorem{corollary}[theorem]{Corollary}
\newtheorem{proposition}[theorem]{Proposition}
\theoremstyle{remark}
\newtheorem{definition}[theorem]{Definition}

\newtheorem{remark}[theorem]{Remark}
\numberwithin{equation}{section}

\newcount\quantno
\everydisplay{\quantno=0}\everycr{\quantno=0}
\newcommand\quant{\advance\quantno by1
                      \ifnum\quantno=1\qquad\else\quad\fi\forall }
\newcommand\itemno[1]{(\romannumeral #1)}

\renewcommand\Re{\operatorname{\mathrm{Re}}}
\renewcommand\Im{\operatorname{\mathrm{Im}}}

\newcommand\grad{\operatorname{\rm{grad}}}
\newcommand\Div{\operatorname{\rm{div}}}
\newcommand\Var{\operatorname{\rm{Var}}}
\newcommand\Lip{\operatorname{\rm{Lip}}}

\newcommand\rest[1]{\kern-.1em
          \lower.5ex\hbox{$\scriptstyle #1$}\kern.05em}

\newcommand\set[1]{{\left\{#1\right\}}}

\renewcommand\mod[1]{\left\vert{#1}\right\vert}
\newcommand\bigmod[1]{\bigl\vert{#1}\bigr|}
\newcommand\Bigmod[1]{\Bigl\vert{#1}\Bigr|}

\newcommand\norm[2]{{\Vert{#1}\Vert_{#2}}}

\newcommand\opnorm[2]{|\!|\!| {#1} |\!|\!|_{#2}}

\newcommand\smallfrac[2]{\mbox{\small$\displaystyle\frac{#1}{#2}$}}
\newcommand\wrt{\,\text{\rm d}}

\newcommand\1{{\bf 1}}

\newcommand\bS{\mathbf{S}} 
\newcommand\bT{\mathbf{T}} 

\newcommand\bW{\mathbf{W}}

\newcommand\BC{\mathbb{C}}

\newcommand\BR{\mathbb{R}}

\newcommand\BZ{\mathbb{Z}}

   \newcommand\fra{\mathfrak{a}}
\newcommand\cB{\mathcal{B}}   
\newcommand\cC{\mathcal{C}}

                               \newcommand\frg{\mathfrak{g}}
   
\newcommand\cI{\mathcal{I}}

\newcommand\cL{\mathcal{L}}   
\newcommand\cM{\mathcal{M}}

\newcommand\cQ{\mathcal{Q}}

\newcommand\cT{\mathcal{T}}

\newcommand\al{\alpha}
\newcommand\be{\beta}
    
\newcommand\de{\delta}
\newcommand\ep{\epsilon}  \newcommand\vep{\varepsilon}

\newcommand\la{\lambda}   
\newcommand\om{\omega}      
\newcommand\si{\sigma}
\newcommand\te{\theta}

\newcommand\OV{\overline}
\newcommand\funnyk{k\hbox to 0pt{\hss\phantom{g}}}

\newcommand\lu[1]{L^1(#1)}
\newcommand\lp[1]{L^p(#1)}

\newcommand\ld[1]{L^2(#1)}
\newcommand\ldO[1]{L^2_0(#1)}

\newcommand\ly[1]{L^\infty(#1)}

\newcommand\hu[1]{H^1(#1)}

\newcommand\wt{\widetilde}
\newcommand\whH{\widehat{\phantom{G}}\hbox to 0pt{\hss $H$}}

\newcommand\emspace{\hbox to 6pt{\hss}}
\newcommand\ds{\displaystyle}

\newcommand\rmi{\hbox{\rm (i)}}
\newcommand\rmii{\hbox{\rm (ii)}}
\newcommand\rmiii{\hbox{\rm (iii)}}
\newcommand\rmiv{\hbox{\rm (iv)}}
\newcommand\rmv{\hbox{\rm (v)}}

\newcommand\ioty{\int_0^{\infty}}

\newcommand\One{{\mathbf{1}}}
\newcommand\OU{Ornstein--Uhlenbeck}

\newcommand\e{\mathrm{e}}

\newcommand\inj{\mathrm{inj}}
\newcommand\diam[1]{\mathrm{diam}#1}

\newcommand\EB{{\rm{(AMP)}}}
\newcommand\PP{{\rm{(I)}}}

\newcommand\supp{\mathrm{supp}\,}

\newcommand\osc{\mathrm{osc}}

\begin{document}

\title[$H^1$ and $BMO$ for nondoubling spaces]
{$H^1$ and $BMO$ for certain nondoubling \\  metric measure  spaces}

\subjclass[2000]{}


\thanks{Work partially supported by the
Progetto Cofinanziato ``Analisi Armonica''.}

\author[A. carbonaro, G. Mauceri and S. Meda]
{Andrea Carbonaro, Giancarlo Mauceri and Stefano Meda}

\address{A. Carbonaro, G. Mauceri: Dipartimento di Matematica \\
Universit\`a di Genova \\ via Dodecaneso 35, 16146 Genova \\ Italia}

\address{E-mail addresses: carbonaro@dima.unige.it \\
mauceri@dima.unige.it}

\address{S. Meda: Dipartimento di Matematica e Applicazioni
\\ Universit\`a di Milano-Bicocca\\
via R.~Cozzi 53\\ 20125 Milano\\ Italy}

\address{E-mail address: stefano.meda@unimib.it}

\begin{abstract}
Suppose that $(M,\rho,\mu)$ is a metric measure space,
which possesses two ``geometric'' properties, called
``isoperimetric'' property and approximate midpoint property,
and that the measure~$\mu$ is locally doubling. The isoperimetric property implies that the volume of balls grows at least exponentially with the radius. Hence the measure $\mu$ is not globally doubling.
In this paper we define an atomic Hardy space $H^1(\mu)$,
where atoms are supported only on ``small balls'',
and a corresponding space $BMO(\mu)$ 
of functions of ``bounded mean oscillation'',
where the control is only on the oscillation over small balls.
We prove that $BMO(\mu)$ is the dual of $H^1(\mu)$ and that an
inequality of John--Nirenberg type 
on small balls holds for functions in $BMO(\mu)$.  
Furthermore, we show that the $\lp{\mu}$ spaces are intermediate 
spaces between $H^1(\mu)$ and $BMO(\mu)$,
and we develop a theory of singular integral operators
acting on function spaces on $M$.
Finally, we show that our theory is strong enough to
give $\hu{\mu}$-$\lu{\mu}$ and $\ly{\mu}$-$BMO(\mu)$
estimates for various interesting operators on Riemannian manifolds
and symmetric spaces which 
are unbounded on $\lu{\mu}$ and on $\ly{\mu}$.
\end{abstract}

\maketitle

\setcounter{section}{0}
\section{Introduction} \label{s:Introduction}

Suppose that $(M,\rho,\mu)$ is a metric measure space.
Assume temporarily that $\mu$ is a doubling measure; then $(M,\rho,\mu)$ 
is a space of homogeneous type in the sense of
Coifman and Weiss.   Harmonic analysis on
spaces of homogeneous type has been the object of many
investigations.  In particular, 
the atomic Hardy space $\hu{\mu}$
and the space $BMO(\mu)$ of functions
of bounded mean oscillation have been defined and studied
in this setting.  We briefly recall their definitions.

An atom $a$ is a function in $\lu{\mu}$ supported in 
a ball $B$ which satisfies appropriate ``size''
and cancellation
condition.  Then $\hu{\mu}$ is the space of
all functions in $\lu{\mu}$ that admit a decomposition
of the form $\sum_j \la_j \, a_j$, where the $a_j$'s
are atoms and the sequence of complex numbers $\{\la_j\}$
is summable.

A locally integrable function $f$ is in $BMO(\mu)$ if 
$$
\sup_B \smallfrac{1}{\mu(B)} \int_B \mod{f-f_B} \wrt \mu < \infty,
$$
where the supremum is taken over \emph{all} balls $B$, and $f_B$
denotes the average of $f$ over $B$.

These spaces enjoy many of the properties of their
Euclidean counterparts.  In particular,
the topological dual of $H^1({\mu})$ is isomorphic to $BMO({\mu})$,
an inequality of John--Nirenberg type holds for functions
in $BMO(\mu)$, the spaces $\lp{\mu}$ are intermediate
spaces between $H^1(\mu)$ and $BMO(\mu)$ for the
real and the complex interpolation methods.
Furthermore, some important operators, which are bounded
on $\lp{\mu}$ for all $p$ in $(1,\infty)$, but otherwise unbounded
on $\lu{\mu}$ and on $\ly{\mu}$, turn out to be bounded
from $H^1({\mu})$ to $\lu{\mu}$ and from $\ly{\mu}$ to $BMO({\mu})$.
We remark that the doubling property is key in 
establishing these results.

There is a huge literature on this subject: we refer the reader
to \cite{CW,St2} and the references therein for further information.

There are interesting cases where $\mu$ is not doubling;
then $\mu$ may or may not be locally doubling.
An important case in which $\mu$ is not even locally
doubling is that of nondoubling measures of polynomial growth
treated, for instance, in \cite{NTV, To, V}, where new
spaces $H^1$ and $BMO$ are defined, and a rich theory is developed
(see also \cite{MMNO} for more general measures on $\BR^n$).
We also mention recent works of X.T.~Duong and L.~Yan
\cite{DY1,DY2}, who define an Hardy space $H^1$ and
a space $BMO$ of bounded mean oscillation
associated to a given operator satisfying suitable estimates.  
This is done in metric measure spaces with
the doubling property, but it is a remarkable fact that
the theory works also for ``bad domains''
in the ambient space, to which the restriction of
 the measure $\mu$ may be nondoubling. 

In this paper we consider the case where $\mu(M) = \infty$,
and $\mu$ is a nondoubling locally doubling measure. 
By this we roughly mean that for every $R$ in $\BR^+$
balls of radius at most $R$ satisfy a doubling condition,
with doubling constant that may depend on $R$ (see (\ref{f: LDC})
in Section~\ref{s: Geom prop} for the precise definition).
Important examples of this situation are complete
Riemannian manifolds with Ricci curvature bounded from below,
a class which includes all Riemannian symmetric spaces of the noncompact type
and Damek--Ricci spaces.
In recent years, analysis on complete Riemannian manifolds
satisfying the local doubling condition has been the object
of many investigations.  For instance, see \cite{Sa} and the references
therein for the equivalence between a scale-invariant parabolic Harnack
inequality and a scaled Poincar\'e inequality,
and \cite{CD,ACDH,Ru} for recents results on the boundedness of 
Riesz transforms on such manifolds.

Our approach to the case of locally doubling measures
is inspired by a result of A.D. Ionescu \cite{I}
on rank one symmetric spaces of the noncompact type
and by a recent paper of the second and third named authors
concerning the analysis of the Ornstein--Uhlenbeck operator \cite{MM}.

For each ``scale'' $b$ in $\BR^+$, we define spaces
$H_b^1(\mu)$ and $BMO_b(\mu)$ much as in the case of spaces
of homogeneous type, the only difference being that
we require that the balls involved have at most radius $b$.
So, for instance, an $H_b^1(\mu)$ atom is an atom supported
in a ball of radius at most $b$. 
We remark that in the case where $M$ is a symmetric space 
of the noncompact type and real rank one, the space $BMO_1(\mu)$
agrees with the space defined by Ionescu.  
Ionescu also proved that
if $p$ is in $(1,2)$, then $\lp{\mu}$ is an interpolation
space between $\ld{\mu}$ and $BMO_1(\mu)$ for the complex
method of interpolation.
In the case where $M$ is a complete noncompact Riemannian manifold
with locally doubling Riemannian measure and satisfying certain
additional assumptions E.~Russ \cite{Ru} defined an Hardy
space that agrees with the space $H_1^1(\mu)$ defined above,
but he did not investigate its structural properties.

We prove that under a mild ``geometric''
assumption, which we call property (AMP) (see Section~\ref{s: Geom prop}), 
$H_b^1(\mu)$ and $BMO_b(\mu)$, in fact, do not depend on the parameter $b$
provided that $b$ is large enough
(see Section~\ref{s: H1}).  Furthermore, we show that
$BMO_b(\mu)$ is isomorphic to the topological dual of $H_b^1(\mu)$
(see Section~\ref{s: duality}), and that functions in $BMO_b(\mu)$ satisfy an
inequality of John--Nirenberg type (see Section~\ref{s: BMO}).

As far as interpolation is concerned, there is no reason to believe
that in this generality $\lp{\mu}$ spaces with $p$ in $(1,\infty)$
are interpolation spaces between
$H_b^1(\mu)$ and $BMO_b(\mu)$.   However, this is true under
a simple geometric assumption on $M$, called property \PP.  
Roughly speaking, $M$ possesses property \PP\ if  
a fixed ratio of the measure of any bounded open set
is concentrated near its boundary.  
If $M$ possesses property \PP, then a basic relative
distributional inequality for the local sharp function
and the local Hardy--Littlewood maximal function holds.
We prove this in Section~\ref{s: sharp},
by adapting to our setting some ideas of Ionescu \cite{I}.
We remark that our approach, which makes use of
the dyadic cubes of M.~Christ and G.~David \cite{Ch,Da}, simplifies
considerably the original proof in \cite{I}.
As a consequence of the relative distributional inequality
we prove an interpolation result for analytic families of operators,
analogous to that proved by C.~Fefferman and E.M.~Stein
\cite{FS} in the classical setting.

An interesting application of the aforementioned interpolation
result is to singular integral operators
(Theorem~\ref{t: singular integrals}).
We prove that if $\cT$ is
a bounded self adjoint operator on $\ld{\mu}$ and its kernel $k$ is
a locally integrable function off the diagonal in $M\times M$
and satisfies a \emph{local H\"ormander type condition}
(i.e. if $\nu_k <\infty$ and $\upsilon_k<\infty$, where $\nu_k$
and $\upsilon_k$ are defined in
the statement of Theorem~\ref{t: singular integrals}), 
then~$\cT$ extends to a bounded
operator on $\lp{\mu}$ for all $p$ in $(1,\infty)$,
from $H^1({\mu})$ to $\lu{\mu}$ and from $\ly{\mu}$ to $BMO({\mu})$.
Applications of this results 
to multipliers for the spherical Fourier transform on
Riemannian symmetric spaces of the noncompact type, 
to spectral operators of the Laplace--Beltrami operator 
and to local Riesz transforms on certain noncompact Riemannian manifolds
are given in Section~\ref{s: Applications}.  Our results
complement earlier results of J.Ph.~Anker \cite{A}, M.~Taylor
\cite{Ta} and Russ \cite{Ru}.

Of course, there are many other interesting operators on measured
metric spaces with properties (AMP), (LDP) and (I),
which are unbounded on $\lu{\mu}$ and on $\ly{\mu}$,
but satisfy $\hu{\mu}$-$\lu{\mu}$ and 
$\ly{\mu}$-$BMO(\mu)$ estimates.   Some of these will be 
considered in a forthcoming paper.

It is interesting to speculate about the range of applicability
of the theory we develop.  In particular, a natural problem
is to find conditions (possibly easy to verify) under which a
complete Riemannian manifold possesses all the three properties,
local doubling, \PP, and (AMP), needed to prove the results
of Sections~\ref{s: Geom prop}-\ref{s: Singular integrals}.
This problem is considered 
in Section~\ref{s: Cheeger}.  Suppose that
$M$ is a complete Riemannian manifold
with Riemannian distance $\rho$ and Riemannian density $\mu$.
A known fact, which is a straightforward consequence
of the Bishop--Gromov comparison Theorem,
is that if $M$ has Ricci curvature bounded from below,
then $(M,\rho,\mu)$ is locally doubling.
Furthermore, since $\rho$ is a length distance, $(M,\rho,\mu)$
has property (AMP).  
We shall prove that $M$ possesses property
\PP\ if and only if the Cheeger isoperimetric constant $h(M)$ 
(see (\ref{f: Cheeger}) for the definition) is strictly positive.  
As a consequence we shall prove that if $M$ has Ricci curvature
bounded from below, then $M$ possesses property~\PP\ 
if and only if the bottom $b(M)$ of the spectrum of 
$M$ is strictly positive.

Similar results on graphs with bounded geometry will
appear elsewhere.

Finally, it would be interesting to consider the case
where $\mu(M) < \infty$.   
To keep the length of this paper reasonable, 
we shall postpone the detailed study of the case where $\mu(M) < \infty$
to a forthcoming paper \cite{CMM}.

\section{Geometric assumptions} \label{s: Geom prop}

Suppose that $(M,\rho,\mu)$ is a metric measure space,
and denote by $\cB$ the family of all balls on $M$.
We assume that $\mu(M) > 0$ and that every ball has finite measure.
For each $B$ in $\cB$ we denote by $c_B$ and $r_B$
the centre and the radius of $B$ respectively.  
Furthermore, we denote by $\kappa \, B$ the
ball with centre $c_B$ and radius $\kappa \, r_B$.
For each $b$ in $\BR^+$, we denote by $\cB_b$ the family of all
balls $B$ in $\cB$ such that $r_B \leq b$.  
For any subset $A$ of $M$ and each $\kappa$ in $\BR^+$
we denote by $A_{\kappa}$ and $A^{\kappa}$ the sets
$$
\bigl\{x\in A: \rho(x,A^c) \leq \kappa\bigr\}
\qquad\hbox{and}\qquad
\bigl\{x\in A: \rho(x,A^c) > \kappa\bigr\}
$$
respectively.

In Sections~\ref{s: Geom prop}-\ref{s: Singular integrals}
we assume that $M$ is \emph{unbounded} 
and possesses the following properties:
\begin{enumerate}
\item[\itemno1]
\emph{local doubling property} (LDP):
for every $b$ in $\BR^+$ there exists a constant $D_b$ 
such that
\begin{equation}  \label{f: LDC} 
\mu \bigl(2 B\bigr)
\leq D_b \, \mu  \bigl(B\bigr)
\quant B \in \cB_b. 
\end{equation}
This property is often called \emph{local doubling condition}
in the literature, and we adhere to this terminology.
Note that if (\ref{f: LDC}) holds and $M$ is bounded, then $\mu$ is doubling.
\item[\itemno2]
\emph{isoperimetric property} \hbox{\PP}:
there exist $\kappa_0$ and $C$ in $\BR^+$ such that
for every bounded open set $A$
\begin{equation}
\label{f: IP}
\mu \bigl(A_{\kappa}\bigr)
\geq C \, \kappa\, \mu (A) \quant \kappa\in (0,\kappa_0]. 
\end{equation}
Suppose that $M$ has property~\PP.  For each $t$ in $(0,\kappa_0]$
we denote by $C_{t}$
the supremum over all constants $C$ for which (\ref{f: IP}) holds for all 
$\kappa$ in $(0,t]$.  Then we define $I_M$ by
$$
I_M = \sup \bigl\{ C_{t}: t\in (0,\kappa_0] \bigr\}.
$$
Note that the function $t\mapsto C_t$ is decreasing on $(0,\kappa_0]$,
so that
\begin{equation} \label{f: Ct decr}
I_M = \lim_{t\to 0^+} C_t;
\end{equation}
\item[\itemno3] 
\emph{property} \hbox{\EB} (approximate midpoint property):
there exist $R_0$ in $[0,\infty)$ and 
$\be$ in $(1/2,1)$ such that for every pair of points~$x$ and $y$
in $M$ with $\rho(x,y) > R_0$ there exists a point $z$ in $M$ such that
$\rho(x,z) < \beta\, \rho (z,y)$ and $\rho(x,y) < \beta\, \rho (x,y)$.

This is clearly equivalent to the requirement that there exists a
ball~$B$ containing~$x$ and $y$ such that $r_B < \be \, \rho(x,y)$.
\end{enumerate}

\begin{remark}
Observe that the isoperimetric property \hbox{\PP}
implies that for every open set $A$ of \emph{finite measure}
$$
\mu \bigl(A_{\kappa}\bigr)
\geq C \, \kappa\, \mu (A) \quant \kappa\in (0,\kappa_0],
$$
where $\kappa_0$ and $C$ are as in (\ref{f: IP}).

Indeed, suppose that $A$ is an open set of finite measure.  
Fix a reference point $o$ in $M$ and denote by $B(o,j)$
the ball with centre $o$ and radius $j$, and by $A(j)$
the set $A \cap  B(o,j)$. 
For each $\kappa$ in $(0,\kappa_0]$ denote by $A_{j,\kappa}$
the set
$$
\bigl\{x \in A(j):
\rho\bigl(x, B(o,j)^c\bigr) \leq \kappa, \rho(x,A^c) > \kappa\bigr\}.
$$
First we prove that 
\begin{equation} \label{f: limit}
\lim_{j\to\infty} \mu \bigl(A_{j,\kappa} \bigr)
= 0.
\end{equation}
Since $\mu(A) < \infty$, for each $\ep >0$ there
exists $J$ such that 
$$
\mu\bigl( A \cap B(o,J)^c\bigr) < \ep.
$$
Now, if $j>J+ \kappa$ and $x$ is in $A_{j,\kappa}$, then $x$
belongs also to $A\cap B(o,J)^c$, whence $\mu(A_{j,\kappa}) < \ep$
for all $j \geq J$, as required. 

Observe that $A_{j,\kappa}$ is contained in $A(j)_{\kappa}$
and that 
$
A_{j,\kappa} 
= A(j)_{\kappa} \setminus \bigl(B(o,j)\cap A_\kappa\bigr)$. 
Therefore
\begin{equation} \label{f: limit II}
\mu\bigl(B(o,j)\cap A_\kappa\bigr)
= \mu\bigl(A(j)_{\kappa}\bigr) - \mu\bigl(A_{j,\kappa}\bigr).
\end{equation}
Since $\mu(A_\kappa) = \lim_{j\to\infty} 
\mu\bigl(B(o,j)\cap A_\kappa\bigr)$, 
$$
\mu(A_\kappa) 
= \lim_{j\to\infty} \mu\bigl(A(j)_{\kappa}\bigr)
$$
by (\ref{f: limit II}) and (\ref{f: limit}).  Since $A(j)$ is a bounded
open set, we may conclude that
$$
\begin{aligned}
\mu(A_\kappa) 
& \geq \lim_{j\to\infty} C\, \kappa \, \mu \bigl(A(j)\bigr) \\
& =    C\, \kappa \, \mu \bigl(A\bigr),
\end{aligned}
$$
as required.
\end{remark}

\begin{remark}
The local doubling property is needed for all the results in this paper,
but many results in Sections~\ref{s: Geom prop}-\ref{s: Singular integrals}
depend only on some but not all the properties \rmi-\rmiii.  
In particular, Theorem~\ref{t: dyadic cubes} and 
Proposition~\ref{p: further prop}
require only the local doubling property,
Propositions~\ref{p: prop PP} and~\ref{p: covering}, Lemma~\ref{l: rdi}
and Theorem~\ref{t: basic},
which are key in proving
the interpolation result Theorem~\ref{t: interpolation},
require property (I), but not property (AMP), all the results in 
Sections~\ref{s: H1}, \ref{s: BMO} and \ref{s: duality} require property
(AMP) but not property (I). 
In particular, property (AMP) is key to prove the scale invariance
of the spaces $\hu{\mu}$ and $BMO(\mu)$ defined below 
(Proposition~\ref{p: decphi}).  
Finally, all the properties \rmi-\rmiii\ above are needed for 
the interpolation results in Section~\ref{s: sharp} and for the results
in Section~\ref{s: Singular integrals}.
\end{remark}

\begin{remark} \label{r: geom I}
The local doubling property implies that for each $\tau \geq 2$
and for each $b$ in $\BR^+$ 
there exists a constant~$C$ such that 
\begin{equation} \label{f: doubling D}
\mu\bigl(B'\bigr)
\leq C \, \mu(B)
\end{equation}
for each pair of balls $B$ and $B'$, with $B\subset B'$, 
$B$ in $\cB_b$, and $r_{B'}\leq \tau \, r_B$.
We shall denote by $D_{\tau,b}$ the smallest constant for 
which (\ref{f: doubling D}) holds. 
In particular, if (\ref{f: doubling D}) holds (with the same constant)
for all balls $B$ in $\cB$, then $\mu$ is doubling and 
we shall denote by $D_{\tau,\infty}$ the smallest constant for 
which (\ref{f: doubling D}) holds. 
\end{remark}

\begin{remark}
There are various ``structural constants'' which enter the proofs
of our results.  We have made the choice to keep track 
of these constants, which often appear explicitly in the statements.
For the reader's convenience we give here a list of all the relevant 
constants used is 
Sections~\ref{s: Geom prop}-\ref{s: Singular integrals}:
$$
\begin{array}{ll}
D_{\tau,b} & \qquad\hbox{$\tau \geq 2$, $b\in \BR^+\cup \{ \infty\}$ 
                   (see Remark~\ref{r: geom I})} \\
I_M        & \qquad\hbox{the isoperimetric constant (see \rmii\ above)} \\
\hbox{$R_0$ and $\be$}  & \qquad\hbox{appear in the (AMP)
                   property (see \rmiii\ above)} \\
\hbox{$\de$, $C_1$ and $a_0$}  & \qquad\hbox{appear 
                   in the construction of dyadic cubes
                   (see Thm~\ref{t: dyadic cubes})}.
\end{array}
$$
\end{remark}

\section{Preliminary results} \label{s: Preliminary}

Roughly speaking, if $M$ has property \PP, then
a fixed ratio of the measure of any bounded open set
is concentrated near its boundary.  The following proposition
contains a quantitative version of this statement.

\begin{proposition} \label{p: prop PP} 
The following hold:
\begin{enumerate}
\item[\itemno1]
the volume growth of $M$ is at least exponential;  
\item[\itemno2] for every bounded open set $A$
$$
\mu(A_t) \geq \bigl(1-\e^{-I_M t}\bigr) \, \mu(A)
\quant t \in \BR^+.
$$
\end{enumerate}
\end{proposition}

\begin{proof}
First we prove \rmi. 
Denote by $o$ a reference point in $M$.
For every $r>0$ denote by~$V_r$ the measure of the ball with
centre $o$ and radius $r$.
It is straightforward to check that 
$B(o,r)_\kappa \subset B(o,r)\setminus B(o,r-\kappa)$,
so that for all sufficiently large $r$
$$
\begin{aligned}
V_r-V_{r-\kappa} 
& \geq \mu\bigl(B(o,r)_\kappa\bigr) \\
& \geq C\, \kappa \, V_r
\quant \kappa \in (0,\kappa_0]
\end{aligned}
$$
by property \PP\ ($C$ and $\kappa_0$ are as in (\ref{f: IP})).  Hence
$$
\begin{aligned}
V_r 
& \geq C\, \kappa \, V_r + V_{r-\kappa} \\
& \geq \eta \, V_{r-\kappa}
\quant \kappa \in (0,\kappa_0],
\end{aligned}
$$
where $\eta=C\kappa + 1$. 
Denote by $n$ the positive integer for which $r-n\kappa$ 
is in $[\kappa,2\kappa)$.  Then 
$$
\begin{aligned}
V_{r}
& \geq \eta^n \, V_{r-n\kappa} \\
& \geq \eta^{r/\kappa-2} \, V_{\kappa}, 
\end{aligned}
$$
as required.

Now we prove \rmii.  Suppose that $A$ is a bounded open subset of $M$,
and that $t$ is in $(0,\kappa_0]$.  Note that $A^s$ is a bounded
open subset of $M$.  It is straightforward to check that
$(A^s)_t \subset A^s\setminus A^{s+t}$.
Therefore
$$
\begin{aligned}
\mu(A^{s+t}) - \mu(A^{s})
& \leq  - \mu\bigl((A^{s})_t\bigr) \\
& \leq  - C_t \, t \, \mu(A^s) 
\end{aligned}
$$
by property \PP\ (see (\ref{f: IP}) above).  
Since $s\mapsto \mu(A^s)$ is monotonic,
it is differentiable almost everywhere.
The inequality above and (\ref{f: Ct decr}) 
imply that for almost every $s$ in $\BR^+$
$$
\frac{\textrm{d}}{\textrm{d} s} \mu(A^s)
\leq -I_M \, \mu(A^s).
$$
Notice also that $\lim_{s\to 0^+} \mu(A^s) = \mu(A)$.  Therefore 
$$
\mu(A^s) \leq \e^{-I_M \, s} \, \mu(A)
\quant s \in \BR^+,
$$
and finally
$$
\mu(A_s) \geq \bigl(1-\e^{-I_M \, s}\bigr) \, \mu(A) \quant s \in \BR^+,
$$
as required. 
\end{proof}

We shall make use of the analogues in our setting
of the so-called dyadic cubes $Q_\al^k$
introduced by G.~David and M.~Christ \cite{Da,Ch} on spaces of
homogeneous type.  It may help to think of $Q_\al^k$
as being essentially a cube of diameter $\de^k$ with centre $z_\al^k$. 

\begin{theorem} \label{t: dyadic cubes}
There exist a collection of open subsets $\{Q_\al^k: k \in \BZ, \al \in I_k\}$
and constants $\de$ in $(0,1)$, $a_0$, $C_1$ in $\BR^+$ such that
\begin{enumerate}
\item[\itemno1]
$\bigcup_{\al} Q_\al^k$ is a set of full measure in $M$
for each $k$ in $\BZ$;
\item[\itemno2]
if $\ell \geq k$, then either $Q_\be^\ell \subset Q_\al^k$ or
$Q_\be^\ell \cap Q_\al^k = \emptyset$;
\item[\itemno3]
for each $(k,\al)$ and each $\ell < k$ there is a unique $\be$
such that $Q_\al^k \subset Q_\be^\ell$;
\item[\itemno4]
$\diam (Q_\al^k) \leq C_1 \, \de^k$;
\item[\itemno5]
$Q_\al^k$ contains the ball $B(z_\al^k, a_0\, \de^k)$.
\end{enumerate}
\end{theorem}

\begin{proof}
The proof of \rmii--\rmv\ is as in \cite{Ch}.  In fact,
the proof depends only on the metric structure of the space
and not on the properties of the measure $\mu$ and is even easier
in our case, because $\rho$ is a genuine distance, 
rather than a quasi-distance.

The proof of \rmi\ is again as in \cite{Ch}; observe that only
a local doubling property is used in the proof.
\end{proof}

Note that \rmiv\ and \rmv\ imply that for every integer $k$
and each $\al$ in $I_k$
$$
B(z_{\al}^k, a_0 \, \de^k)
\subset Q_{\al}^k
\subset B(z_{\al}^k, C_1 \, \de^k).
$$
\begin{remark}
When we use dyadic cubes, we implicitly assume that for each
$k$ in $\BZ$ the set $M \setminus \bigcup_{\al \in I_k} Q_{\al}^k$
has been permanently deleted from the space.
\end{remark}

We shall denote by $\cQ^k$ the class of
all dyadic cubes of ``resolution'' $k$, i.e., the family of
cubes $\{Q_{\al}^k: \al \in I_k\}$, and by $\cQ$ 
the set of all dyadic cubes.
We shall need the following additional properties of dyadic cubes.

\begin{proposition} \label{p: further prop}  
Suppose that $b$ is in $\BR^+$ and that $\nu$ is in $\BZ$, and let 
$C_1$ and $\de$ are as in Theorem~\ref{t: dyadic cubes}.  
The following hold:
\begin{enumerate}
\item[\itemno1]
suppose that $Q$ is in $\cQ^k$ for some $k\geq \nu$, and that
$B$ is a ball such that $c_B \in Q$. 
If $r_B \geq C_1\, \de^k$, then  
\begin{equation} \label{f: ineq II}
\mu(B\cap Q) = \mu(Q);
\end{equation}
if $r_B < C_1\, \de^k$, then  
\begin{equation} \label{f: ineq I}
\mu(B\cap Q) 
\geq D_{C_1/(a_0\de),\de^\nu}^{-1} \, \mu(B);
\end{equation}
\item[\itemno2]
suppose that $\tau$ is in $[2,\infty)$.  For each $Q$ in $\cQ$ the 
space $\bigl(Q,\rho_{\vert Q}, \mu_{\vert Q}\bigr)$
is of homogeneous type.  Denote by $D_{\tau,\infty}^Q$ 
its doubling constant (see Remark~\ref{r: geom I} for the definition).  
Then
$$
\sup \, \Bigl\{D_{\tau,\infty}^Q:  Q \in \bigcup_{k=\nu}^\infty \cQ^k \Bigr\}
\leq D_{\tau, C_1\de^\nu}\, D_{C_1/(a_0\de), \de^\nu};
$$ 
\item[\itemno3]
for each ball $B$ in $\cB_b$, let $k$ be the integer
such that $\de^k \leq r_B < \de^{k-1}$, and
and let~$\wt B$ denote the ball with
centre $c_B$ and radius $\bigl(1+C_1\bigr)\, r_B$. 
Then $\wt B$ contains all dyadic cubes in $\cQ^k$ that
intersect $B$ and 
$$
\mu(\wt B) \leq D_{1+C_1,b} \, \mu(B);
$$
\item[\itemno4]
suppose that $B$ is in $\cB_b$, and that $k$ is an integer such that
$\de^k \leq r_B < \de^{k-1}$.  Then there are at most 
$D_{(1+C_1)/(a_0\de),b}$ dyadic cubes in $\cQ^k$ that intersect $B$.
\end{enumerate}
\end{proposition}

\begin{proof}
First we prove \rmi.
Our proof is a version of the proof
given by Christ \cite[p.~613]{Ch} that keeps track
of the various structural constants involved.

First we prove (\ref{f: ineq II}).
By Theorem~\ref{t: dyadic
cubes}~\rmiv\ the diameter of $Q$
is at most $C_1\, \de^k$, so that $Q \subset B$, whence $B\cap Q = Q$,
and the required formula is obvious.

To prove (\ref{f: ineq I}),
denote by $j$ the unique integer such that
$$
\de^j < \frac{r_B}{C_1} \leq \de^{j-1}
$$
and by $Q_\be^j$ the unique dyadic
cube of resolution $j$ that contains $c_B$.  Then $j\geq k$, because
$C_1 \, \de^j< r_B \leq C_1 \de^k$.
The cubes $Q_\be^j$ and $Q$ have nonempty 
intersection, because, they both contain $c_B$.
Thus $Q_\be^j \subset Q$.  By Theorem~\ref{t: dyadic
cubes}~\rmiv\ the diameter of $Q_\be^j$
is at most $C_1\, \de^j$, which is $< r_B$ by the definition of $j$,
so that $Q_\be^j \subset B$.
Therefore $Q_\be^j \subset B\cap Q$, and
$$
\begin{aligned}
\mu(B\cap Q)
& \geq \mu\bigl(Q_\be^j\bigr) \\
& \geq \mu\bigl(B(z_\be^j,a_0\, \de^j)\bigr). 
\end{aligned}
$$
Observe that
$
\smallfrac{r_B}{a_0\, \de^j}
\leq \smallfrac{C_1}{a_0\, \de}
$ 
and that $B(z_\be^j,a_0\, \de^j) \subset B$.
Hence 
$$
\mu\bigl(B(z_\be^j,a_0\, \de^j)\bigr) 
\geq D_{C_1/(a_0\de),\de^\nu}^{-1} \, \mu(B), 
$$
as required to conclude the proof of (\ref{f: ineq I}), and of \rmi.

Next we prove \rmii.
Suppose that $Q$ is a dyadic cube in $\cQ^k$,
with $k \geq \nu$.  Suppose that $B$ and $B'$ are balls in $\cB$
with $B\subset B'$
such that $c_B$ and $c_{B'}$ belong to $Q$ and $r_{B'} \leq \tau \, r_B$.  
We treat the cases where $C_1\, \de^k \leq r_B$ and $r_B < C_1 \, \de^k$
separately.

If $C_1\, \de^k \leq r_B$, then $\mu(B'\cap Q) = \mu(Q) = \mu(B\cap Q)$.

If $r_B < C_1 \, \de^k$, then 
$$
\begin{aligned}
\mu(B'\cap Q)
& \leq \mu(B') \\
& \leq D_{\tau, C_1\de^k} \, \mu(B) \\
& \leq D_{\tau, C_1\de^k}\, D_{C_1/(a_0\de), \de^\nu} \, \mu(B\cap Q),
\end{aligned}
$$ 
by the local doubling property of $M$ and (\ref{f: ineq I}).
Therefore $Q$ is a homogeneous space with doubling constant
at most $D_{\tau, C_1\de^k}\, D_{C_1/(a_0\de), \de^\nu}$.
Since $k\geq \nu$, these doubling constants are dominated by
$D_{\tau, C_1\de^\nu}\, D_{C_1/(a_0\de), \de^\nu}$,
as required.

Now we prove \rmiii.  
Denote by $Q$ a cube
in $\cQ^k$ that intersects $B$.  By the triangle inequality
and Theorem~\ref{t: dyadic cubes}~\rmiv, $Q$  
is contained in $\wt B$.  The required estimate of
the measure of $\wt B$ follows from the local doubling condition
(see Remark~\ref{r: geom I}).

Finally we prove \rmiv.
Denote by $Q_1,\ldots, Q_N$ the cubes
in $\cQ^k$ that intersect $B$.  By \rmiii\ 
each of these cubes is contained in $\wt B$.
By Theorem~\ref{t: dyadic cubes}~\rmv\
each cube $Q_j$ contains a ball, $B_j'$ say, of radius~$a_0\, \de^k$,
and these balls are pairwise disjoint because they are contained
in disjoint dyadic cubes.
By the local doubling condition $\mu(\wt B) \leq D_{(1+C_1)/(a_0\de),b} \, \mu(B_j')$
for all $j$.  Therefore
$$
\begin{aligned}
N \, \mu(\wt B)
& \leq D_{(1+C_1)/(a_0\de),b} \, \sum_{j=1}^N \mu(B_j') \\
& =    D_{(1+C_1)/(a_0\de),b} \, \mu \Bigl(\bigcup_{j=1}^N B_j' \Bigr) \\
& \leq D_{(1+C_1)/(a_0\de),b} \, \mu(\wt B),
\end{aligned}
$$
from which the desired estimate follows.
\end{proof}

Our next result is a covering property enjoyed by 
spaces with property \PP.  It is key in proving 
relative distributional inequalities for the sharp maximal operator
(see Lemma~\ref{l: rdi} below).

\begin{proposition} \label{p: covering}
Suppose that $\nu$ is an integer.
For every $\kappa$ in $\BR^+$, every 
open subset $A$ of~$M$ of finite measure and every collection~$\cC$ of
dyadic cubes of resolution at least $\nu$ 
such that $\bigcup_{Q\in \cC} Q =A$, there exist
mutually disjoint cubes $Q_1,\ldots,Q_k$ in $\cC$ such that
\begin{enumerate}
\item[\itemno1]
$\sum_{j=1}^k \mu (Q_j) \geq \bigl(1-\e^{-I_M\, \kappa}\bigr)   
\, \mu (A)/2$;
\item[\itemno2]
$ \rho (Q_j, A^c) \leq \kappa$
for every $j$ in $\{1,\ldots, k\}$.
\end{enumerate}
\end{proposition}

\begin{proof}
Denote by $\wt\cC$ the subcollection of all cubes in $\cC$
that intersect $A_{\kappa}$.  Clearly the cubes in $\wt\cC$ cover $A_{\kappa}$
and satisfy \rmii.

Next we prove \rmi.
Since two dyadic cubes are either disjoint or 
contained one in the other, we may consider the sequence
$\{Q_j\}$ of cubes in $\wt \cC$
which are not contained in any other cube of $\wt\cC$.  
The existence of these ``maximal'' cubes is
guaranteed by the assumption that the resolution of
the cubes in $\cC$ is bounded from below.  The cubes
$\{Q_j\}$ are mutually disjoint and cover $A_\kappa$.  Therefore
\begin{equation} \label{f: A tilde}
\begin{aligned}
\mu(A)
& \geq \sum_{j} \mu (Q_j) \\
& \geq \mu \bigl(A_{\kappa}\bigr) \\
& \geq \bigl(1-\e^{-I_M\, \kappa}\bigr)\, \mu  (A),
\end{aligned}
\end{equation}
where the last inequality holds because
$M$ possesses property \PP.
To conclude the proof of~\rmi\ take $k$ so large that
$\sum_{j=1}^k \mu (Q_j)\geq (1/2) \sum_{j=1}^\infty \mu (Q_j)$.  Then
$$
\sum_{j=1}^k \mu (Q_j)
\geq \bigl(1-\e^{-I_M\, \kappa}\bigr) \, \mu (A)/2,
$$
as required.
\end{proof}

\section{The Hardy space $H^1$} \label{s: H1}

\begin{definition} \label{d: standard atom}
Suppose that $r$ is in $(1,\infty]$.  A $(1,r)$-\emph{atom} $a$
is a function in $\lu{\mu}$ supported in a ball $B$ in $\cB$
with the following properties:
\begin{enumerate}
\item[\itemno1]
$\norm{a}{\infty}  \leq \mu (B)^{-1}$
if $r$ is equal to $\infty$ and
$$
\Bigl(\frac{1}{\mu (B)} \int_B \mod{a}^r \wrt{\mu } \Bigr)^{1/r}
\leq \mu (B)^{-1}
$$
if $r$ is in $(1,\infty)$;
\item[\itemno2]
$\ds \int_B a \wrt \mu  = 0$.
\end{enumerate}
\end{definition}

\begin{definition} \label{d: Hardy}
Suppose that $b$ is in $\BR^+$.  
The \emph{Hardy space} $H_b^{1,r}({\mu})$ is the 
space of all functions~$g$ in $\lu{\mu}$
that admit a decomposition of the form
\begin{equation} \label{f: decomposition}
g = \sum_{k=1}^\infty \la_k \, a_k,
\end{equation}
where $a_k$ is a $(1,r)$-atom \emph{supported in a ball $B$ of $\cB_b$},
and $\sum_{k=1}^\infty \mod{\la_k} < \infty$.
The norm $\norm{g}{H_b^{1,r}({\mu})}$
of $g$ is the infimum of $\sum_{k=1}^\infty \mod{\la_k}$
over all decompositions (\ref{f: decomposition})
of $g$.  
\end{definition}

Clearly a function in $H_c^{1,r}(\mu)$ is in $H_b^{1,r}(\mu)$
for all $c<b$.  We shall prove in Proposition~\ref{p: decphi} below
that, in fact, the reverse
inclusion holds whenever $c$ is large enough.  
Hence for $b$ large the space $H_{b}^{1,r}(\mu)$
does not depend on the parameter~$b$, and for each pair of sufficiently large
parameters $b$ and~$c$ the norms $\norm{\cdot}{H_b^{1,r}(\mu)}$
and $\norm{\cdot}{H_c^{1,r}(\mu)}$ are equivalent.

There are cases where 
$H_c^{1,r}(\mu)$ and $H_b^{1,r}(\mu)$ are isomorphic spaces
for each pair of parameters $c$ and $b$.  This happens, for
instance, if $M$ is the upper half plane, $\rho$
the Poincar\'e metric and $\mu$ the associated Riemannian measure.
However, if $M$ is a homogeneous tree of degree $q\geq 1$,
$\rho$ denotes the natural distance and $\mu$ the counting measure, 
it is straightforward to check that $H_1^1(\mu)$ consists of the null
function only, whereas $H_2^1(\mu)$ is a much richer space. 

Recall that $R_0$ and $\be$ are the constants which appear in the definition
of the (AMP) property.

\begin{proposition}\label{p: decphi} 
Suppose that $r$ is in $(1,\infty]$, 
$b$ and $c$ are in $\BR^+$ and satisfy $R_0/(1-\be)<c<b$.
The following hold:
\begin{enumerate}
\item[\itemno1]
there exist a constant $C$ and a nonnegative integer $N$, 
depending only on $M$, $b$ and $c$, such that for each 
ball $B$ in $\cB_b$ and each $(1,r)$-atom $a$ supported in
$B$ there exist at most $N$ atoms $a_1,\ldots,a_N$
with supports contained in balls $B_1,\ldots,B_N$ in $\cB_{c}$ 
and $N$ constants $\la_1,\ldots,\la_N$ such that 
$\mod{\la_j}\leq C$,
$$
a
= \sum_{j=1}^N \la_j \, a_j
\qquad\hbox{and}\qquad
\norm{a}{H_c^{1,r}(\mu)}
\le C \, N;
$$
\item[\itemno2]
a function $f$ is in $H_{c}^{1,r}(\mu)$ if and only
if $f$ is in $H_{b}^{1,r}(\mu)$.  Furthermore,
there exists a constant $C$ such that  
$$
\norm{f}{H_{b}^{1,r}(\mu)} 
\leq \norm{f}{H_{c}^{1,r}(\mu)} 
\leq C \, \norm{f}{H_{b}^{1,r}(\mu)} 
\quant f \in H_{c}^{1,r}(\mu).
$$
\end{enumerate}
\end{proposition}

\begin{proof}
Choose $\be'$ in $(0,1-\be)$ such that $R_0/\be' <c$.

First we prove \rmi.
Suppose that $B$ is a ball in $\cB_b$ and that
$r_B > c$, for otherwise there is nothing to prove.    
Denote by $\{z_1,\ldots, z_{N_1}\}$ a maximal
set of points in $B$ such that $\rho(z_j,z_k) \geq \be' r_B$
for all $j\neq k$ 
and each point of $B$ is at distance at most $\be' r_B$
from the set $\{z_1,\ldots, z_{N_1}\}$.  Denote by $B_j$
the ball with centre $z_j$ and radius $\be' r_B$, and
by $B_0$ the ball with centre $c_B$ and
radius $\be' r_B$.
Note that 
\begin{equation} \label{f: dcentouno}
\mu(B)
\leq D_{1/\be',b}\, \mu(B_0),
\end{equation}
where $D_{1/\be',b}$ is as in Remark~\ref{r: geom I}.
We consider the partition of unity $\{\psi_1,\ldots,\psi_{N_1}\}$
of $\One_{\bigcup_j B_j}$
subordinated to the covering $\{B_1,\ldots, B_{N_1}\}$ defined by
$
\psi_j 
= \One_{B_j}/\sum^{N_1}_{k=1}\One_{B_k}.
$
For each $j$ in $\set{1,\dots,{N_1}}$ we define
$$
A_j
= \frac{1}{\mu(B_0)}   \int_{M} 
a \, \psi_j \wrt \mu 
\qquad\hbox{and}\qquad
\phi_j
= a\, \psi_j - A_j\, \One_{B_0}.
$$
It is straighforward to check that 
$
a = \sum_{j=1}^{N_1} \phi_j
$
and each function $\phi_j$ has integral 
$0$ and is supported in $B_j\cup B_0$.  Define $J'$ and $J''$ by 
$$
J' = \{j: z_j\notin B_0 \}
\qquad\hbox{and}\qquad
J'' = \{j: z_j\in B_0 \}.
$$

If $j$ is in $J'$, then $z_j$ is not in $B_0$, so that
$\rho(c_B,z_j) > \be' r_B > \be' c > R_0$.
Therefore we may use property \EB\ and conclude that there exists
a ball $B_j'$ containing $c_B$ and $z_j$ with radius $< \be \,\rho(c_B,z_j)$.
Denote by $\wt B_j$ the ball centred at $c_{B_j'}$ with radius
$r_{B_j'} + \be' r_B$.  By using the triangle inequality we see
that $\wt B_j$ contains $B_j\cup B_0$.  
Observe that $r_{\tilde B_j} \leq (\be+ \be') \, r_B$, which
is strictly less than $r_B$, because we assumed that $\be' < 1- \be$. 

Next we check that if $j$ is in $J'$, then
$\phi_j$ is a multiple of a $(1,r)$-atom:
we give details in the case where $r=2$; the cases
where $r\in (1,\infty)\setminus\{2\}$ may be treated similarly,
and the variations needed to treat the case where $r=\infty$ are
straightforward and are omitted.
By the triangle inequality 
$$
\begin{aligned}
\Bigl(\smallfrac{1}{\mu(\wt B_j)} \int_{\wt B_j} \mod{\phi_j}^2 
\wrt \mu \Bigr)^{1/2}
& \leq   \Bigl(\smallfrac{1}{\mu(\wt B_j)} \int_{\wt B_j} \mod{a\, \psi_j}^2 
      \wrt \mu \Bigr)^{1/2} 
      + \mod{A_j} \Bigl(\smallfrac{1}{\mu(\wt B_j)} 
      \int_{\wt B_j} \One_{B_0} \wrt \mu \Bigr)^{1/2} \\
& \leq  \sqrt{\smallfrac{\mu(B)}{\mu(\wt B_j)}}  
      \,\Bigl(\smallfrac{1}{\mu(B)} \int_{B} \mod{a}^2 
      \wrt \mu \Bigr)^{1/2} \
       + \sqrt{\smallfrac{\mu(B_0)}{\mu(\wt B_j)}}  
      \,\smallfrac{1}{\mu(B_0)} \int_{M} \mod{a\,\psi_j} 
      \wrt \mu  \\
& \leq  \Bigl(\sqrt{\smallfrac{\mu(B)}{\mu(B_0)}}  
       + \smallfrac{\mu(B)}{\mu(B_0)} \Bigr) \, \smallfrac{1}{\mu(B)} \\
& \leq  \smallfrac{2\, D_{1/\be',b}}{\mu(B)}.
\end{aligned}
$$
Observe that
$$
\smallfrac{1}{\mu(B)}
\leq \smallfrac{1}{\mu(B_0)}
\leq \smallfrac{D_{\be/\be'+1,b}}{\mu(\wt B_j)},
$$
because $B_0$ is contained both in $B$ and $\wt B_j$ and 
the ratio between the radii of $\wt B_j$ and $B_0$ is at most 
$\be/\be'+1$.
Therefore we may conclude that
$$
\Bigl(\smallfrac{1}{\mu(\wt B_j)} \int_{\wt B_j} \mod{\phi_j}^2 
\wrt \mu \Bigr)^{1/2}
\leq  \frac{2\, D_{1/\be',b} \, D_{\be/\be'+1,b}}{\mu(\wt B_j)},
$$
i.e., $\phi_j/(2\, D_{1/\be',b} \, D_{\be/\be'+1,b})$ is an atom supported
in the ball $\wt B_j$ of radius at most $(\be+\be') \, r_B$.   

Now suppose that $j$ is in $J''$.  Then $B_j \cup B_0$ is contained
in $2B_0$.  Notice that $\be'<1-\be<1/2$, so that
$r_{2B_0} = 2\be' r_B < (\be+\be')\, r_B$.   
By arguing much as above, we see that
$\phi_j/(2\, D_{1/\be',b} \, D_{\be/\be'+1,b})$ is an atom supported
in the ball $2B_0$ of radius $<(\be+\be')\, r_B$.

We have written $a$ as the sum of $N_1$ functions $\phi_j$,
each of which is a multiple of an atom with constant $2\, D_{1/\be',b} \, D_{\be/\be'+1,b}$.
Thus, we have proved that
$\norm{a}{H_{b(\be+\be')}^{1,r}} \leq 2\, D_{1/\be',b} 
\, D_{\be/\be'+1,b} \, {N_1}$.

Now, if $j$ is in $J'$, and $r_{\tilde B_j}<c$, then
$\wt B_j$ is in $\cB_{c}$.  Similarly, if $j$ is 
in $J''$, and $r_{2B_0}<c$, then $2B_0$ is in $\cB_c$.
If $2B_0$ and all the balls $\wt B_j$ are in $\cB_c$, then
the proof is complete.  
Otherwise either $2B_0$ or some of the $\wt B_j$'s 
is not in $\cB_{c}$, and we must iterate 
the construction above.  It is clear that
after a finite number of steps, depending 
on the ratio $b/c$, we end up with the required decomposition.

Next we prove \rmii.
Obviously $\norm{f}{H_{b}^{1,r}}\leq \norm{f}{H_{c}^{1,r}}$,  
so we only have to show that 
$\norm{f}{H_{c}^{1,r}}\leq C\, \norm{f}{H_{b}^{1,r}}$
for some constant $C$ depending only on $b$ and $c$ and $M$.  
But this follows directly from \rmi.
\end{proof}

\begin{definition}
Suppose that $r$ is in $(1,\infty)$. 
Then for every $b$ and $c$ in $\BR^+$
such that $R_0/(1-\be)<c<b$ the spaces $H_{b}^{1,r}(\mu)$ and 
$H_{c}^{1,r}(\mu)$ are isomorphic (in fact, they contain the same
functions) by Proposition~\ref{p: decphi}~\rmii, 
and will simply be denoted by $H^{1,r}(\mu)$.
\end{definition}

Later (see Section~\ref{s: duality}) 
we shall prove that $H^{1,r}(\mu)$ does not depend on
the parameter $r$ in $(1,\infty)$, and we shall denote $H^{1,r}(\mu)$
simply by $\hu{\mu}$.

\section{The space $BMO$} \label{s: BMO}

Suppose that $q$ is in $[1,\infty)$.  
For each locally integrable function~$f$ we define $N_b^q(f)$ by
$$
N_b^q(f)
=  \sup_{B\in\cB_b} \Bigl(\frac{1}{\mu(B)}
\int_B \mod{f-f_B}^q \wrt\mu \Bigr)^{1/q},
$$
where $f_B$ denotes the average of $f$ over $B$. 
We denote by $BMO_b^q(\mu)$ the space of all equivalence classes
of locally integrable functions $f$ modulo constants,
such that $N_b^q(f)$ is finite, endowed with the norm
$$
\norm{f}{BMO_b^q(\mu)}
= N_b^q(f).
$$

Notice that only ``small'' balls enter in the definition of $BMO_b^q(\mu)$.  
It is a nontrivial fact, proved in Proposition~\ref{p: normeq} below,
that~$BMO_b^q(\mu)$ is independent of the parameter 
$b$, provided $b$ is large enough, 
and that the norms $N_b^q$ are all equivalent.

\begin{proposition}\label{p: normeq}
Suppose that
$q$ is in $[1,\infty)$, and 
$b$ and $c$ are positive constant such that $R_0/(1-\be)<c<b$.
Then $BMO_b^q(\mu)$ and
$BMO_c^q(\mu)$ coincide and the norms $N_b^q$ and $N_c^q$ are equivalent. 
\end{proposition}

\begin{proof} 
Obviously, if $0<c<b$ and $f$ is in $BMO_b^q(\mu)$, 
then $f$ is in $BMO_c^q(\mu)$ and $N_c^q(f)
\leq N_b^q(f)$.  Thus, we only have to show that 
$N_b^q(f) \le C \, N_c^q(f)$ 
for some constant $C$ depending only on $b$ and $c$ and $M$.  
We give the proof in the case where $q=1$; the proof in
the other cases is similar. 

Suppose that $B$ is a ball in $\cB_b$.  Observe that
\begin{align*}
\frac{1}{\mu (B)}\int_B \mod{f-f_B}\wrt \mu 
& \le \frac{2}{\mu (B)} \, \inf_{c\in\BC}\int_B \mod{f-c}\wrt \mu \\
& \le \frac{2}{ \mu (B)}\, \norm{f}{L^1(B)/\BC} , 
\end{align*}
where ${L^1(B)/\BC}$ is the quotient of the space 
$L^1(B)$ modulo the constants. 
Since the dual of ${L^1(B)/\BC}$ is $L^\infty_0(B)$ (the space
of all functions in $\ly{B}$ with vanishing integral, endowed
with the $\ly{B}$ norm),
$$
\norm{f}{L^1(B)/\BC} 
= \sup \Bigl\{ \Bigmod{\int_B f\, \phi \, \wrt  \mu }: 
\phi\in L^\infty_0(B), \norm{\phi}{\infty}\le 1  \Bigr\}.
$$
Suppose that $\phi$ is a function in $L^\infty_0(B)$ 
with $\norm{\phi}{\infty}\le 1$.   Then $\phi/\mu(B)$
is a $(1,\infty)$-atom and,  by Proposition~\ref{p: decphi}~\rmi,
there exist $(1,\infty)$-atoms $a_1,\ldots,a_N$ 
supported in balls $B_j$ in $\cB_{c}$ whose union contains $B$ such that 
$\phi/\mu(B) = \sum_{j=1}^N \la_j\, a_j$,
with $\mod{\la_j} \leq C$ and $\norm{a_j}{\infty}\le 1/\mu(B_j)$, where 
$C$ and $N$ are constants which depend only on $b$, $c$ and $M$.  Thus
\begin{align*}
\frac{1}{ \mu (B)}\Bigmod{\int_B f\, \phi\wrt \mu }
& \le C\sum_{j=0}^N \, \int_{B_j} 
      \bigmod{f-f_{B_j}} \, \mod{a_j} \wrt \mu  \\ 
& \le C \, \sum_{j=0}^N 
      \, \frac{1}{ \mu (B_j)}\int_{B_j}\mod{f-f_{B_j}} \wrt \mu  \\ 
& \le C\, N \, N_c^1(f).
\end{align*}
Hence $N_b^1(f) \leq 2\,C \,N\, N_c^1(f)$, as required. 
\end{proof}

\begin{remark}
For the rest of this paper,
we fix a constant $b_0 > R_0/(1-\be)$.
For each $q$ in $[1,\infty)$ we denote by~$BMO^q(\mu)$
the space $BMO_{b_0}^q(\mu)$ endowed with any of the equivalent
norms $N_b^q$, where $b>R_0/(1-\be)$.  
\end{remark}

Next, we want to show that ~$BMO^q(\mu)$ is independent of 
$q$ (see the remark at the end of this section).
The strategy is the same as in the classical case: it hinges on
a John--Nirenberg type inequality for functions in $BMO^1(\mu)$.
The original inequality was proved in \cite{JN}, where classical
functions of bounded mean oscillation appeared for the first time. 
We need the following generalization of the John--Nirenberg 
inequality to doubling spaces which is stated in  \cite{CW} 
and proved in \cite[Thm~2.2]{Buc} (see also \cite{MMNO, MP}).

\begin{proposition}\label{p: JNDMS}
Suppose that $(X,d,\mu)$ is a doubling metric measure space, with
doubling constant $D$. 
There exist constants $J_D$ and $\eta_D$, 
which depend only on $D$, such that
$$
\mu\bigl(\bigl\{ x \in B: \bigmod{f(x) - f_B} > s\bigr\} \bigr)
\le  J_D\, \e^{-\eta_D s/\norm{f}{BMO(X)}} \, \mu(B)
\quant s \in \BR^+ \quant B.
$$
\end{proposition}

By Proposition~\ref{p: further prop}~\rmii\ 
for each dyadic cube $Q$ the measured
metric space $(Q,\rho_{\vert Q}, \mu_{\vert Q})$ is a space of
homogeneous type.  We denote by $BMO(Q)$ the classical 
$BMO$ space on $Q$.

\begin{theorem} \label{t: JNI}
Denote by $\nu$ the unique integer such that $\de^\nu \leq b_0 < \de^{\nu-1}$,
and by $N$ the norm $N_{2\max(C_1,b_0)}^1$ on $BMO(\mu)$.
The following hold:
\begin{enumerate}
\item[\itemno1]
for each dyadic cube $Q$ in $\bigcup_{k = \nu}^\infty \cQ^k$
and for each~$f$ in $BMO^1(\mu)$
the restriction of $f$ to $Q$ is in $BMO(Q)$ and 
$$
\norm{f}{BMO(Q)}
\leq  2\, D_{C_1/(a_0\de),b_0\max(1,a_0)}\, N(f);
$$
\item[\itemno2]
there exist positive constants $J$ and $\eta$ such that for every 
function $f$ in $BMO^1(\mu)$ and for every ball $B$ in $\cB_{b_0}$
$$
\mu\bigl(\{x\in B: \mod{f(x)-f_B} >s \} \bigr)
\leq J \, \e^{-\eta \,s/N(f) } \, \mu(B).
$$
\end{enumerate}
\end{theorem}

\begin{proof}
First we prove \rmi.  Suppose that $Q$ is in $\cQ^k$.
Recall that a ball in $Q$ is the intersection of~$Q$ with a ball $B$
in $\cB$ whose centre belongs to $Q$.  
We have to estimate the oscillation $\osc_f(B\cap Q)$
of $f$ over $B\cap Q$ defined by
\begin{equation} \label{f: osc on BQ}
\osc_f(B\cap Q)
= \frac{1}{\mu(B\cap Q)} \int_{B\cap Q} \bigmod{f-f_{B\cap Q}} \wrt \mu.
\end{equation}
We shall prove that 
$$
\osc_f(B\cap Q) 
\leq 2\, D_{C_1/(a_0\de),b_0\max(1,a_0)} \, N(f)
\quant f \in BMO^1(\mu),
$$
from which \rmi\ follows.
We consider the cases where $r_B < C_1\,\de^k$ and 
$r_B \geq C_1\, \de^k$ separately.

In the case where $r_B < C_1\,\de^k$ we compare (\ref{f: osc on BQ})
with the oscillation of $f$ over $B$.  
By the triangle inequality 
$$
\begin{aligned}
\osc_f(B\cap Q)
& \leq \frac{1}{\mu(B\cap Q)} \int_{B\cap Q} \bigmod{f-f_{B}} \wrt \mu
     + \bigmod{f_B-f_{B\cap Q}} \\
& \leq \frac{2}{\mu(B\cap Q)} \int_{B\cap Q} \bigmod{f-f_{B}} \wrt \mu. 
\end{aligned}
$$
By Proposition~\ref{p: further prop}~\rmi\ 
we know that $\mu(B\cap Q) \geq D_{C_1/(a_0\de),\de^\nu}^{-1} \, \mu(B)$;
hence the right hand side in the displayed formula above
may be estimated from above~by 
$$
\frac{2\, D_{C_1/(a_0\de),\de^\nu}}{\mu(B)} 
\int_{B} \bigmod{f-f_{B}} \wrt \mu,
$$
which, in turn, may be majorised by 
$D_{C_1/(a_0\de),b_0} \, N(f)$.

Now suppose that $r_B \geq C_1\, \de^k$.
Since $\diam (Q) < C_1 \, \de^{k}$ by Theorem~\ref{t: dyadic cubes}~\rmiv, 
$Q\cap B=Q$.  For the sake of definitess, suppose that $Q$ is 
the dyadic cube $Q_{\be}^k$.  Then $Q_\be^k$ contains the
ball $B(z_{\be}^k,a_0\, \de^k)$.  Denote by $\wt B$ the ball
centred at $z_\be^k$ and radius $C_1\, \de^k$.
Now, 
$$
\begin{aligned}
\osc_f(B\cap Q)
& \leq \frac{1}{\mu(B\cap Q)} \int_{B\cap Q} \bigmod{f-f_{\wt B}} \wrt \mu
     + \bigmod{f_{\wt B}-f_{B\cap Q}} \\
& \leq \frac{2}{\mu(B\cap Q)} \int_{B\cap Q} 
       \bigmod{f-f_{\wt B}} \wrt \mu \\
& \leq \frac{2}{\mu\bigl(B(z_{\be}^k,a_0\, \de^k)\bigr)} 
       \int_{\wt B} \bigmod{f-f_{\wt B}} \wrt \mu \\
& \leq \frac{2\, D_{C_1/a_0,a_0\de^\nu}}{\mu\bigl(\wt B\bigr)} 
       \int_{\wt B} \bigmod{f-f_{\wt B}} \wrt \mu,
\end{aligned}
$$
which is majorised by $2\, D_{C_1/a_0,a_0b_0} \, N(f)$.
The proof of \rmi\ is complete.

Now we prove \rmii.  Suppose that $B$ is $\cB_{b_0}$. 
Denote by $k$ the unique
integer such that $\de^k \leq r_B < \de^{k-1}$ and by
$Q_1,\ldots,Q_N$ the dyadic cubes of resolution $k$
that intersect $B$.  By Proposition~\ref{p: further prop}~\rmiv\ we have the
estimate $N\leq D_{(1+C_1)/(a_0\de),b_0}$.
Then
\begin{equation} \label{f: claim JN}
\mu\bigl(\{x\in B: \mod{f(x)-f_B} > s \} \bigr)
\leq \sum_{j=1}^N \mu\bigl(\{x\in Q_j: \mod{f(x)-f_B} > s \} \bigr).
\end{equation}
We estimate each of the summands on the right hand side from 
above by
\begin{equation} \label{f: claim JN II}
\mu\bigl(\{x\in Q_j: \mod{f(x)-f_{Q_j}} > s/2 \} \bigr)
+ \mu\bigl(\{x\in Q_j: \mod{f_B-f_{Q_j}} > s/2 \} \bigr).
\end{equation}
By Proposition~\ref{p: JNDMS} and \rmi\ the first summand 
in this formula is majorised by 
$$
J_{Q_j}\, \e^{-\eta_{Q_j} s/\norm{f}{BMO(Q_j)}} \, \mu(Q_j)
\leq J_{Q_j}\, \e^{-\eta_{Q_j} s/(2 D_{C_1/(a_0\de),b_0\max(1,a_0)}  
N(f))} \, \mu(Q_j).
$$
Here we use the fact that since $\diam(Q_j)$ is finite, then
$Q_j$ is a ball in the doubling space $(Q_j,\rho_{\vert Q_j},
\mu_{\vert Q_j})$.
By Proposition~\ref{p: further prop}~\rmii\ 
all the spaces $(Q_j,\rho_{\vert Q},\mu_{\vert Q})$ 
are spaces of homogeneous type with doubling constant
dominated by $D_{\tau, C_1b_0}\, D_{C_1/(a_0\de), b_0}$,
which we simply denote by $D'$.  
Also, denote by $\eta'$ the constant 
$\eta_{D'}/(2 D_{C_1/(a_0\de),b_0\max(1,a_0)})$.
By Proposition~\ref{p: further prop}~\rmiii\ the ball $\wt B$
with centre $c_B$
and radius $(1+C_1) \, r_B$ contains $Q_1,\ldots,Q_N$ and
$\mu(\wt B) \leq D_{C_1+1,{b_0}} \, \mu(B)$.
Thus, by summing over $j$, we see that
\begin{equation} \label{f: prima}
\begin{aligned}
\sum_{j=1}^N \mu\bigl(\{x\in Q_j: \mod{f(x)-f_{Q_j}} > s/2 \} \bigr)
&  \leq J_{D'}
        \, \e^{-\eta' s/N(f)} \, \sum_{j=1}^N \mu(Q_j)  \\
&  \leq J_{D'}
        \, \e^{-\eta' s/N(f)} \, \mu(\wt B) \\
&  \leq J_{D'}
        \, \e^{-\eta' s/N(f)} D_{C_1+1,b_0} \, \mu(B).
\end{aligned}
\end{equation}

Now we estimate the second summand in (\ref{f: claim JN II}).
We claim that 
$$
\mod{f_B-f_{Q_j}} \leq 2\, D_{(1+C_1)/(a_0 \de), b_0} \, N(f).
$$

Indeed, 
$$
\begin{aligned}
\mod{f_B-f_{Q_j}}
& \leq \mod{f_B-f_{\wt B}} + \mod{f_{\wt B}-f_{Q_j}} \\
& \leq \frac{1}{\mu(B)} \int_B \mod{f-f_{\wt B}} \wrt \mu
       + \frac{1}{\mu(Q_j)} \int_{Q_j} \mod{f-f_{\wt B}} \wrt \mu \\
& \leq \frac{1}{\mu(B)} \int_B \mod{f-f_{\wt B}} \wrt \mu
       + \frac{1}{\mu\bigl(B(z_j,a_0\de^k)\bigr)} 
         \int_{\wt B} \mod{f-f_{\wt B}} \wrt \mu \\
& \leq \frac{2\, D_{(1+C_1)/(a_0 \de), b_0}}{\mu(\wt B)} 
         \int_{\wt B} \mod{f-f_{\wt B}} \wrt \mu,
\end{aligned}
$$
which is dominated by $2\, D_{(1+C_1)/(a_0 \de), b_0}\, N(f)$, as claimed.

Thus
$$
\mu\bigl(\{x\in Q_j:  \mod{f_B -f_{Q_j}} > s/2 \} \bigr)
\leq \mu\bigl(\{x\in Q_j: 2 D_{(1+C_1)/(a_0 \de), b_0}
\, N(f)  >s /2 \} \bigr).
$$
Now, the right hand side is equal to $\mu(Q_j)$ when $s$ is in 
$\bigl(0,4D_{(1+C_1)/(a_0 \de), b_0} \, N(f) \bigr)$,
and to~$0$ when $s$ is in 
$\bigl[4D_{(1+C_1)/(a_0 \de), b_0} \, N(f) , \infty \bigr)$,
so that 
$$
\mu\bigl(\{x\in Q_j:  \mod{f_B -f_{Q_j}} > s/2 \} \bigr)
\leq \e^{4D_{(1+C_1)/(a_0 \de), b_0}} \, \e^{-s/N(f)  } \, \mu(Q_j)
\quant  s \in \BR^+.
$$
Therefore
\begin{equation} \label{f: seconda}
\begin{aligned}
\sum_{j=1}^N \mu\bigl(\{x\in Q_j:  \mod{f_B -f_{Q_j}} > s/2 \} \bigr)
& \leq \e^{4D_{(1+C_1)/(a_0 \de), b_0}} \, \e^{-s/N(f)  } \, \mu(\wt B) \\
& \leq \e^{4D_{(1+C_1)/(a_0 \de), b_0}} \,D_{C_1+1,b_0}\,
      \e^{-s/N(f)  } \, \mu(B).
\end{aligned}
\end{equation}
Now, (\ref{f: prima}) and (\ref{f: seconda}) imply that
$$
\begin{aligned}
\mu\bigl(\{x\in B: \mod{f(x)-f_B} >s \} \bigr)
&  \leq \bigl(J_{D'} \, \e^{-\eta' s/N(f)}  
    + \e^{4D_{(1+C_1)/(a_0 \de), b_0}}\, \e^{-s/N(f)}\bigr) \,D_{C_1+1,b_0} 
       \, \mu(B) \\
&\leq J \, \e^{-\eta \,s/N(f) } \, \mu(B),
\end{aligned}
$$
where $J = \bigl(J_{D'}+\e^{4D_{(1+C_1)/(a_0 \de), b_0}}\bigr)
\,D_{C_1+1,b_0}$ and $\eta = \min (1,\eta')$, as required.
\end{proof}

A standard consequence of the John--Nirenberg type inequality
is the following.

\begin{corollary} \label{c: JNI}
Denote by $\nu$ the unique integer such that $\de^\nu \leq b_0 < \de^{\nu-1}$,
and by $N$ the norm $N_{2\max(C_1,b_0)}^1$ on $BMO(\mu)$.
The following hold:
\begin{enumerate}
\item[\itemno1]
for every $c<\eta$
$$
\int_B \e^{c\, \mod{f-f_B}/N(f)} 
\wrt\mu \leq \Bigl(1+\frac{Jc}{\eta-c}\Bigr)\, \mu(B)
\quant f \in BMO(\mu) \quant B \in \cB_{b_0},
$$
where $\eta$ and $J$ are as in Theorem~\ref{t: JNI}~\rmii;
\item[\itemno2]
for each $q$ in $(1,\infty)$ there exists a constant $C$ such that
$$
\Bigl(\frac{1}{\mu(B)}
\int_B \mod{f-f_B}^q \wrt\mu \Bigr)^{1/q}
\leq C\, N(f)
\quant f \in BMO(\mu) \quant B \in \cB_{b_0}.
$$
\end{enumerate}
\end{corollary}

\begin{proof}
First we prove \rmi.
Observe that the left hand side of \rmi\ is equal to 
$$
\mu(B) + \int_1^\infty \mu\left(\{x\in B: \mod{f-f_B} > 
N(f) \, (\log \be)/c \}\right) \wrt \be.  
$$
Changing variables and using the John--Nirenberg type
inequality proved in Theorem~\ref{t: JNI} we see that the last integral
may be estimated by
$$
\mu(B)  \Bigl[1+J \, \ioty \e^{(c-\eta) v/c} \wrt v \Bigr].
$$
The above integral is finite if and only if $c < \eta$ and it is equal to
$c/(\eta-c)$: the required inequality follows.

Now we prove \rmii.  
By elementary calculus, for each $q$ in $(1,\infty)$ there exists
a constant~$C_q$ such that $e^s\geq C_q \, s^q$ for every $s$ in $\BR^+$.
Therefore \rmi\ implies that
$$
C_q\, \Bigl(\frac{c}{N(f)}\Bigr)^q   
\int_B \, \mod{f-f_B}^q \wrt\mu 
\leq \Bigl(1+\frac{Jc}{\eta-c}\Bigr)\, \mu(B),
$$
which is equivalent to the required estimate.

The proof of the corollary is complete.
\end{proof}

\begin{remark} \label{r: BMO}
By Corollary~\ref{c: JNI}~\rmii, if $f$ is in $BMO^1(\mu)$, then $f$
is in $BMO^q(\mu)$ for all $q$ in $(1,\infty)$. 
Conversely, if $f$ is in $BMO^q(\mu)$ for some $q$ in $(1,\infty)$, 
then trivially it is in $BMO^1(\mu)$, hence in $BMO^r(\mu)$ for all $r$
in $(1,\infty)$ by Corollary~\ref{c: JNI}~\rmii.
Furthermore, the norms $N_{b_0}^1$ and $N_{b_0}^q$ are equivalent. 
In view of this observation, all spaces $BMO^q(\mu)$, $q$ in $[1,\infty)$,
coincide.  We shall denote $BMO^1(\mu)$ simply by $BMO(\mu)$.  
We endow $BMO(\mu)$ with any of the equivalent norms $N_b^q$,
where $q$ is in $[1,\infty)$ and $b>R_0/(1-\be)$.
This remark will be important in the proof of the duality
between the Hardy space $H^1(\mu)$ and $BMO(\mu)$ (see
Section~\ref{s: duality} below).
\end{remark}

\section{Duality} \label{s: duality}

We shall prove that the topological dual of $H^{1,r}({\mu})$ 
may be identified with $BMO^{r'}({\mu})$, where $r'$ 
denotes the index conjugate to $r$.
Suppose that $1<r<s<\infty$.  Then 
$\bigl(H^{1,r}({\mu})\bigr)^* = \bigl(H^{1,s}({\mu})\bigr)^*$, 
because we have proved that $BMO^{r'}({\mu})= BMO^{s'}({\mu})$
(see Remark~\ref{r: BMO}).  
Observe that the identity is a continuous injection
of $H^{1,s}(\mu)$ into $H^{1,r}(\mu)$, and that $H^{1,s}(\mu)$
is a dense subspace of $H^{1,r}(\mu)$.  Then we may conclude
that $H^{1,s}(\mu) = H^{1,r}(\mu)$.

We need some more notation and some preliminary observation.
For each ball $B$ in $\cB_{b_0}$ let $\ldO{B}$
denote the Hilbert space of all functions $f$ in $\ld{\mu}$ such that
the support of $f$ is contained in $B$ and $\int_B f \wrt {\mu } = 0$.
We remark that a function $f$ in $\ldO{B}$ is a multiple
of a $(1,2)$-atom, and that
\begin{equation}  \label{f: basic estimate atom}
\norm{f}{H^{1,2}({\mu})} \leq \mu (B)^{1/2} \, \norm{f}{\ld{B}}.
\end{equation}

Suppose that $\ell$ is a 
bounded linear functional on $H^{1,2}({\mu})$.  Then, for
each $B$ in $\cB_{b_0}$ the restriction of $\ell$ to $\ldO{B}$ is
a bounded linear functional on $\ldO{B}$.  Therefore, by the Riesz
representation theorem there
exists a unique function $\ell^B$ in $\ldO{B}$ which represents
the restriction of $\ell$ to $\ldO{B}$.
Note that for every constant $\eta$ the function $\ell^B+\eta$
represents the same functional, though it is not in $\ldO{B}$
unless $\eta$ is equal to $0$.  Denote by $\opnorm{\ell}{H^{1,2}(\mu)}$
the norm of $\ell$. 
Observe that
\begin{equation} \label{f: ellB}
\begin{aligned}
\norm{\ell^B}{\ldO{B, {\mu}}}
& = \sup_{\norm{f}{\ldO{B}=1}} \Bigmod{\int_B \ell^B \, f \wrt {\mu }} \\
& \leq \sup_{\norm{f}{\ldO{B}=1}} \, \opnorm{\ell}{H^{1,2}({\mu})} \,
     \norm{f}{H^{1,2}({\mu})}  \\
& \leq \mu (B)^{1/2} \, \opnorm{\ell}{H^{1,2}({\mu})},
\end{aligned}
\end{equation}
the last inequality being a consequence of (\ref{f: basic estimate atom}).

For every $f$ in $BMO^{r'}({\mu})$ and every finite linear combination
$g$ of $(1,r)$-atoms the integral $\int_{\BR^d} f\, g \wrt{\mu }$ is convergent.
Denote by $H_{\mathrm{fin}}^{1,r}({\mu})$ the subspace of $H^{1,r}({\mu})$
consisting of all finite linear combinations of $(1,r)$-atoms.
Then $g \mapsto \int_{\BR^d} f\, g \wrt {\mu }$ defines a linear
functional on $H_{\mathrm{fin}}^{1,r}({\mu})$.  We observe that
$H_{\mathrm{fin}}^{1,r}({\mu})$ is dense in $H^{1,r}({\mu})$.

\begin{theorem} \label{t: duality}
Suppose that $r$ is in $(1,\infty)$.  The following hold
\begin{enumerate}
\item[\itemno1]
for every $f$ in $BMO^{r'}({\mu})$ the functional $\ell$, initially defined
on $H_{\mathrm{fin}}^{1,r}({\mu})$ by the rule
$$
\ell(g) = \int_{\BR^d} f\, g \wrt {\mu },
$$
extends to a bounded functional on $H^{1,r}({\mu})$.  Furthermore,
$$
\opnorm{\ell}{H^{1,r}({\mu})}
\leq \norm{f}{BMO^{r'}({\mu})};
$$
\item[\itemno2]
there exists a constant $C$ such that for every continuous linear functional
$\ell$ on $H^{1,r}({\mu})$ there exists a function $f^\ell$ in $BMO^{r'}({\mu})$ such that
$\norm{f^\ell}{BMO^{r'}({\mu})} \leq C \, \opnorm{\ell}{H^{1,r}({\mu})}$ and
$$
\ell(g) = \int_{\BR^d} f^\ell\, g \wrt {\mu }
\quant g \in H_{\mathrm{fin}}^{1,r}({\mu}).
$$
\end{enumerate}
\end{theorem}

\begin{proof}
The proof of \rmi\ follows the line of the proof of \cite{CW}
which is based on the classical result of C.~Fefferman \cite{F, FS}.
We omit the details.

Now we prove \rmii\ in the case where $r$ is equal to $2$.
The proof for $r$ in $(1,\infty)\setminus\{2\}$
is similar and is omitted.

Recall that for each $b> R_0/(1-\be)$ 
the space $H^{1,2}(\mu)$ is isomorphic to $H_b^{1,2}(\mu)$ 
with norm $\norm{\cdot}{H_b^{1,2}(\mu)}$.
Thus, we may interpret $\ell$ as a continuous linear functional
on $H_b^{1,2}(\mu)$ for each $b>R_0/(1-\be)$.  
Fix a reference point $o$ in $M$. 
For each $b$ there exists a function $f_b^\ell$ in 
$L^2_{0}(B(o,b))$ that represents $\ell$ in $B(o,b)$.
Since both $f_1^\ell$ and the restriction of $f_b^\ell$
represent $\ell$ on $B(o,1)$, there exists a constant
$\eta_b$ such that  
$$
f_1^\ell-f_b^\ell = \eta_b
$$
on $B(o,1)$.  By integrating both sides of this equality on $B(o,1)$
we see that 
$$
\eta_b 
= -\frac{1}{\mu\bigl(B(o,1)\bigr)} \int_{B(o,1)} f_b^\ell \wrt \mu.
$$
Define 
$$
f^{\ell} (x) = f_b^{\ell}(x) + \eta_b
\quant x \in B(o,b) \quant b \in [1,\infty).
$$ 
It is straightforward to check that this is a good definition.
We claim that the function $f^\ell$ is
in $BMO({\mu})$ and there exists a constant~$C$ such that
$$
\norm{f^{\ell}}{BMO({\mu})}
\leq C \, \opnorm{\ell}{H^{1,2}({\mu})^*}
\quant \ell \in H^{1,2}({\mu})^*.
$$
Indeed, choose a ball $B$ in $\cB_{b_0}$.  
Then there exists a constant $\eta^B$ such that
\begin{equation} \label{f: fell on B}
f^\ell \big\vert_B = \ell^B + \eta^B,
\end{equation}
where $\ell^B$ is in $\ldO{B}$ and represents the 
restriction of $\ell$ to $\ldO{B}$.
By integrating both sides on $B$, we see that $\eta^B = \bigl(f^\ell)_B$.
Then, by (\ref{f: fell on B}),
$$
\begin{aligned}
\Bigl( \frac{1}{\mu (B)} \int_B \bigmod{f^\ell - \bigl(f^\ell\bigr)_B}^2
\wrt {\mu } \Bigr)^{1/2}
& = \Bigl( \frac{1}{\mu (B)} \int_B \bigmod{\ell^B}^2
      \wrt {\mu } \Bigr)^{1/2} \\
& \leq \opnorm{\ell}{H^{1,2}({\mu})}, 
\end{aligned}
$$
so that $N_{b_0}^2(f^\ell) \leq \opnorm{\ell}{H^{1,2}({\mu})}$,
as required.
\end{proof}

\begin{remark}
Note that the proof of Theorem~\ref{t: duality} does not
apply, strictly speaking, to the case where $r$ is equal to
$\infty$.  However, a straightforward, though tedious, adaptation to the case
where $\mu$ is only locally doubling
of a classical result \cite{CW}, show that $H^{1,\infty}(\mu)$
and $H^{1,2}(\mu)$ agree, with equivalence of norms. 
Consequently, the dual space of $H^{1,\infty}(\mu)$ is $BMO(\mu)$. 
\end{remark}

\section{Estimates for the sharp function and interpolation}
\label{s: sharp}

The main step in the proof of Fefferman--Stein's interpolation
result for analytic families of operators is a certain
relative distributional inequality (also referred to as ``good $\la$
inequality'' in the literature) \cite[Thm~5, p.~153]{FS}, \cite{St2},
which is a modified
version of the original relative distributional inequality of D.L.~Burkholder
and R.F.~Gundy \cite{BG, Bur} for martingales.

Extensions of Fefferman--Stein's distributional inequality
to spaces of homogeneous type are available in the literature
(see, e.g., Mac{\'\i}as' thesis \cite{Ma}).  It may be worth
observing that the doubling property plays a key r\^ole in their
proof.
An extension of this theory to rank one symmetric spaces of the
noncompact type is due to Ionescu \cite{I}.
In this section we adapt Ionescu's ideas and arguments to
our setting.

For each integer $k$, and each locally
integrable function $f$,
the \emph{noncentred dyadic local Hardy--Littlewood maximal function}
$\cM_kf$ is defined by
\begin{equation} \label{f: HL}
\cM_k f(x)
= \sup_{Q} \frac{1}{\mu (Q)} \int_Q \mod{f} \wrt \mu 
\quant x \in M,
\end{equation}
where the supremum is taken over all dyadic cubes of 
resolution $\geq k$ that contain $x$.

For each $p$ in $M$ we denote by  $\cB_b(p)$ 
the subcollection of all balls in $\cB_b$ which contain $p$.
For each $b$ in $\BR^+$ we define a \emph{local sharp function}
$f^{\sharp,b}$ of a locally integrable function $f$ thus:
$$
f^{\sharp,b}(p)
= \sup_{B\in\cB_b (p)} \frac{1}{\mu (B)} \int_B \mod{f-f_B} \wrt \mu 
\quant p \in M.
$$
Observe that $f$ is in $BMO(\mu)$ if and only if $\norm{f^{\sharp,b}}{\infty}$
is finite for some (hence for all) $b$ in $(R_0/(1-\be),\infty)$.

We shall need the following result, whose proof, \emph{mutatis mutandis},
is the same as that of its Euclidean analogue.

\begin{theorem} \label{t: local HL}
Suppose that $k$ is an integer.  Then the noncentred dyadic local
Hardy--Littlewood maximal operator
$\cM_k$ is bounded on $\lp{\mu}$ for every $p$ in $(1,\infty]$ and of
weak type~$1$.
\end{theorem}

\begin{lemma} \label{l: rdi}
Define constants $C_0$, $b'$, $\si$ and $D$ by
$$
C_0 = \max(C_1/\de,\de),
\quad
b' = \max(b_0, 2 C_1+C_0),
\quad
\si = \bigl(1-\e^{-I_M\, \de^3}\bigr)/2
\quad\hbox{and}\quad
D =D_{b'/a_0,a_0}, 
$$
where $a_0$, $C_1$ and $\de$ are as in Theorem~\ref{t: dyadic cubes},
and $D_{b'/a_0,a_0}$ is defined in Remark~\ref{r: geom I}.
For every $\eta'$ in $(0,1)$, for all  positive $\vep<(1-\eta')/(2D)$,
and for every $f$ in $\lu{\mu}$
$$
\mu \bigl(\{ \cM_{2}f >\al, \, f^{\sharp,b'} \leq \vep\, \al\}\bigr)
\leq \eta \, \mu \bigl(\{ \cM_{2}f >\eta'\, \al\}\bigr)
\quant \al \in \BR^+,
$$
where 
$$
\eta=
1-\si + \smallfrac{2\vep \,D}{\si \,(1-\eta')}.
$$
\end{lemma}

\begin{proof}
For each $\be>0$ we denote by $A(\be)$ and $S(\be)$ the sets 
$\{\cM_{2}f>\be\}$ and $\{f^{\sharp,b'} >\be\}$ respectively.
The inequality to prove may then be rewritten as follows:
$$
\mu \bigl(A({\al})\cap S({\vep\al})^c \bigr)
\leq \eta\, \mu \bigl(A({\eta'\al}) \bigr)
\quant \al\in \BR^+.
$$

To each $x$ in $A({\eta'\al})$ we associate the maximal
dyadic cube $Q_x$ containing $x$ of resolution at least~$2$ 
such that $\mod{f}_{Q_x} > \eta'\al$.  Here $\mod{f}_{Q_x}$ 
denotes the average of $\mod{f}$ on the cube $Q_x$.
We denote by $\cC_{\eta'\al}$ 
the collection of cubes $\{Q_x\}_{x\in A({\eta'\al})}$. 
Clearly $A(\eta'\al) = \bigcup_{x\in A(\eta'\al)} Q_x$,
and $\mu\bigl(A(\eta'\al)\bigr)<\infty$, because $\cM_2$ is of
weak type 1.
By Proposition~\ref{p: covering} (with $\kappa = \de^3$)
there exist mutually disjoint cubes
$Q_1,\ldots,Q_k$ in $\cC_{\eta'\al}$ such that
$\rho\bigl(Q_{j},A(\eta'\al)^c\bigr) \leq \de^3$ and
\begin{equation} \label{f: estimate from below}
\sum_{j=1}^k \mu (Q_j) \geq \si \, \mu \bigl(A(\eta'\al)\bigr).
\end{equation}

We claim that if $0<\vep<(1-\eta')/(2D)$ then
\begin{equation} \label{f: claim balls}
\mu \bigl(Q_j\cap A(\al) \cap S(\vep\al)^c \bigr)
\leq 
\smallfrac{2\vep \,D}{\si \,(1-\eta')}
\, \mu \bigl(Q_j \bigr)
\quant j \in \{1,\ldots,k\}.
\end{equation}
We postpone for a moment the proof of the claim and show how
(\ref{f: claim balls}) implies the required conclusion.
Observe that $A(\al)  \subset A(\eta'\al)$ and that
$$
\begin{aligned}
\mu \bigl(A(\al) \cap S(\vep\al)^c \bigr)
& = \mu \Bigl(\bigl(A(\eta'\al) \setminus
       \bigcup_{j=1}^k Q_j\bigr)\cap A(\al) \cap S(\vep\al)^c \Bigr)
    + \mu \Bigl(\bigl(\bigcup_{j=1}^k Q_j\bigr)
        \cap A(\al) \cap S(\vep\al)^c \Bigr) \\
& \leq (1-\si) \, \mu \bigl(A(\eta'\al)\bigr) + 
    \smallfrac{2\vep \,D}{\si \,(1-\eta')} \, \sum_{j=1}^k \mu (Q_j) \\
& \leq \eta \, \mu \bigl(A(\eta'\al)\bigr).  \\
\end{aligned}
$$
The penultimate inequality is a consequence of (\ref{f: estimate from below})
and of (\ref{f: claim balls}), and
the last inequality follows from the fact that the $Q_j$'s are mutually
disjoint cubes contained in $A(\eta'\al)$.

Thus, to conclude the proof of the lemma it remains to
prove the claim (\ref{f: claim balls}).
For the rest of the proof
we shall denote any of the cubes $Q_1,\ldots,Q_k$ simply by $Q$.
Denote by $\nu_0$ the resolution of $Q$. 

We claim that there exists a dyadic cube $\wt {Q}$ of 
the same resolution as $Q$ such that 
$\mod{f}_{\wt {Q}}\le\eta'\al$ 
and $\rho (Q,\wt {Q})\le C_0\, \de^{\nu_0}$.

We treat the cases where $\nu_0$ is equal to $2$ or $\geq 3$
separately.

Suppose that $\nu_0=2$.
Since $\rho\big(Q,A(\eta'\al)^c)\big)\le  \de^3$, 
there exists a point $y$ in $A(\eta'\al)^c$ such that 
$\rho(Q,y)\le  \de^3$.  Denote by $\wt {Q}$ the dyadic
cube with resolution $2$ which contains $y$. 
Then $\rho(Q,\wt {Q})\le \de^3\le C_0 \de^2$ 
and $\mod{f}_{\wt {Q}}\le\eta'\al$, because 
$\wt {Q}\cap A(\eta'\al)^c\not=\emptyset$.

Now suppose that $\nu_o\ge3$.  Then the father $Q^\sharp$ 
of $Q$ contains a point $y$ in $A(\eta'\al)^c$,
for otherwise $Q^\sharp$ would be contained in $A(\eta'\al)$,
thereby contradicting the maximality of $Q$. 
Denote by $\wt {Q}$ the dyadic cube of resolution $\nu_0$ which contains
$y$. 
Then $y$ is in $\wt {Q}\cap Q^\sharp$ 
and therefore $\wt {Q}\subset Q^\sharp$. 
Thus 
$$
\rho(Q,\wt {Q})
\le \, \diam(Q^\sharp)
\le C_1 \,  \de^{\nu_0-1}
\le C_0  \, \de^{\nu_0}
$$ 
and $\mod{f}_{\wt {Q}}\le\eta'\al$, 
because $\wt {Q}\cap A(\eta'\al)^c\not=\emptyset$.
This completes the proof of the claim. 

To each point $y$ in $Q\cap A(\al)$ we associate a maximal
dyadic cube $Q_y'$ of resolution at least $2$
containing $y$ such that $\mod{f}_{Q_y'}>\al$.
Denote $\cC'$ the collection of all these cubes.
By Proposition~\ref{p: covering} we may select mutually disjoint
cubes $Q_1',\ldots,Q_{k'}'$ in $\cC'$ such that
\begin{equation} \label{f: measure B'}
\sum_{j=1}^{k'} \mu (Q_j')
\geq \si \, \mu \bigl(Q\cap A(\al)\bigr).
\end{equation}
Note also that $Q_1',\ldots,Q_{k'}'$ are contained in $Q$. 
Denote by $B^*$ a ball with centre at a point of $Q$ 
and radius $b'\de^{\nu_0}$ (recall that $b'\geq 2C_1+C_0)$. 
Then $B^*$ contains both $Q$, whence the cubes $Q'_1,\ldots,Q'_{k'}$, 
and $\widetilde{Q}$.  Hence 
\begin{equation}\label{B^*Q}
\mu(B^*)\le D\, \mu(Q)\qquad\hbox{and} \qquad
\mu(B^*)\le D\, \mu(\widetilde{Q})
\end{equation}
by the local doubling property. 

If $Q\cap A(\al)\cap S(\vep\al)^c$ is nonempty, then
\begin{equation} \label{f: if nonempty}
\int_{B^*} \mod{f-f_{B^*}} \wrt \mu 
\leq \vep \, \al\, \mu (B^*).
\end{equation}
Since $\wt Q \subset B^*$ and $\mod{f}_{\wt Q}\leq \eta'\,\al$,
$$
\mu (\wt Q) \, \bigl(\mod{f_{B^*}} - \eta' \, \al\bigr)
\leq \int_{B^*} \mod{f-f_{B^*}} \wrt \mu 
$$
by the triangle inequality.
Now (\ref{f: if nonempty}) implies that

\begin{equation} \label{f: eliminating I}
\mu (\wt Q) \, \bigl(\mod{f_{B^*}} - \eta' \, \al\bigr)
\leq \vep \, \al\, \mu (B^*).
\end{equation}
By a similar argument
\begin{equation} \label{f: eliminating II}
\bigl(\al- \mod{f_{B^*}} \bigr) \sum_{j=1}^{k'} \mu (Q_j') \,
\leq \vep \, \al\, \mu (B^*).
\end{equation}
Fr{}om (\ref{f: eliminating I}) we see that
$\mod{f_{B^*}} \leq \al \, \bigl(\eta'+\vep \,
\smallfrac{\mu (B^*)}{\mu (\wt Q)}\bigr)$.
By inserting this inequality in (\ref{f: eliminating II}),
we obtain that
$$
\Bigl(1-\eta'-\vep\,\frac{\mu (B^*)}{\mu (\wt Q)} \Bigr)
\sum_{j=1}^{k'} \mu (Q_j')
\leq \vep \, \mu (B^*),
$$
whence
$$
\si \, \Bigl(1-\eta'-\vep\,\frac{\mu (B^*)}{\mu (\wt Q)} \Bigr)
\, \mu (Q\cap A(\al) \cap S(\vep\al)^c)
\leq \vep \, \mu (B^*),
$$
by (\ref{f: measure B'}).  
Now, since $\vep<(1-\eta')/(2D)$, we may use (\ref{B^*Q}) and conclude that
$$
\, \mu (Q\cap A(\al) \cap S(\vep\al)^c)
\leq \frac{2\vep \,D}{\si \,(1-\eta')}\, \mu (Q),
$$
as required.
\end{proof}

\begin{theorem} \label{t: basic}
For each $p$ is in $(1,\infty)$
there exists a positive constant $C$ such that
$$
\norm{f^{\sharp,b'}}{\lp{\mu}}
\geq C\, \norm{f}{\lp{\mu}}
\quant f \in \lp{\mu},
$$
where $b' = \max( b_0, 2C_1+C_0)$ is as in the statement of 
Lemma~\ref{l: rdi}.
\end{theorem}

\begin{proof}
Observe that it suffices to show that
\begin{equation} \label{f: sharp max}
\norm{f^{\sharp,b'}}{\lp{\mu}}
\geq C\, \norm{\cM_{2}f}{\lp{\mu}},
\end{equation}
because $\cM_{2}f \geq \mod{f}$ by the differentiation theorem
of the integral, which is a standard consequence of Proposition~\ref{t:
local HL}.

Let $\eta$ and $\eta'$ be as in the statement of Lemma~\ref{l: rdi}.
By Lemma~\ref{l: rdi}
$$
\begin{aligned}
\norm{\cM_{2}f}{\lp{\mu}}^p
& = p \ioty \al^{p-1}\, \mu \bigl(A(\al) \bigr)\wrt \al  \\
& = p \ioty \al^{p-1}
        \, \mu \bigl(A(\al) \cap S(\vep\al)^c\bigr )\wrt \al
      + p \ioty \al^{p-1} \, \mu \bigl(A(\al) \cap S(\vep\al)\bigr)
        \wrt \al \\
& \leq  p\, \eta \ioty \al^{p-1}
        \, \mu \bigl(A(\eta'\al)\bigr) \wrt \al
      + p \ioty \al^{p-1}
        \, \mu \bigl(S(\vep\al)\bigr )\wrt \al \\
& =  p\, \eta\, \eta'^{-p} \ioty \be^{p-1}
        \, \mu \bigl(A(\be)\bigr )\wrt \be
      + p\,  \vep^{-p} \ioty \be^{p-1}
        \, \mu \bigl(S(\be)\bigr )\wrt \be \\
& \leq  \eta\, \eta'^{-p} \, \norm{\cM_{2}f}{\lp{\mu}}^p + \vep^{-p} \,
        \norm{f^{\sharp,b'}}{\lp{\mu}}^p . \\
\end{aligned}
$$
Now, for a given $p$, we choose $\eta'$ such that $\eta'^p = 1 - \si/4$,
and then we choose $\vep$ small enough so that $\eta \leq 1-\si/2$.
Therefore $\eta\, \eta'^{-p}<1$ and (\ref{f: sharp max}) follows.
\end{proof}

As a consequence of the relative
distributional inequality proved in Theorem \ref{t: basic}
we establish an interpolation result for
analytic families of operators.  
We need the following notation.  In the following, when
$X$ and $Y$ are Banach spaces, and $\te$ is in $(0,1)$,
we denote by $(X,Y)_{[\te]}$ the complex interpolation space
between $X$ and $Y$ with parameter $\te$.

\begin{theorem} \label{t: interpolation}
Suppose that $\te$ is in $(0,1)$.  The following hold:
\begin{enumerate}
\item[\itemno1]
if $p_\te$ is $2/(1-\te)$, then
$\bigl(\ld{\mu},BMO(\mu)\bigr)_{[\te]} = L^{p_\te}(\mu)$;
\item[\itemno2]
if $p_\te$ is $2/(2-\te)$, then
$\bigl(\hu{\mu},\ld{\mu}\bigr)_{[\te]} = L^{p_\te}(\mu)$.
\end{enumerate}
\end{theorem}

\begin{proof}
The proof of \rmi\ is an adaptation of the proof of 
\cite[Cor.~1, p.156]{FS}.  We omit the details.

Now \rmii\ follows from \rmi\ and the duality theorem
\cite[Corollary~4.5.2]{BL}.  
\end{proof}

\begin{theorem}
Denote by $S$ the strip $\{z\in \BC: \Re z\in (0,1)\}$.
Suppose that $\{\cT_z\}_{z\in \bar S}$ is a family of uniformly bounded
operators on $\ld{\mu}$ such that $z\mapsto \int_{\BR^d} \cT_zf \, g \wrt \mu $
is holomorphic in $S$ and continuous in $\bar S$ for all
functions $f$ and $g$ in $\ld{\mu}$.
Further, assume that there exists a constant $A$ such that
$$
\opnorm{\cT_{is}}{\ld{\mu}} \leq A
\qquad\hbox{and}\qquad
\opnorm{\cT_{1+is}}{\ly{\mu};BMO({\mu})} \leq A.
$$
Then for every $\te$ in $(0,1)$ the operator $\cT_\te$ is bounded
on $L^{p_\te}({\mu})$, where $p_\te = 2/(1-\te)$, and
$$
\opnorm{\cT_\te}{L^{p_\te}({\mu})}
\leq A_\te,
$$
where $A_\te$ depends only on $A$ and on $\te$.
\end{theorem}

\begin{proof}
This follows directly from \rmi\ of the previous theorem
and \cite[Thm~1]{CJ}.  Alternatively,
we may follow the line of the proof of the classical case
(see, for instance, \cite[Thm~4, p.175]{St2}, or \cite{FS}).
\end{proof}

\section{Singular integrals} \label{s: Singular integrals}

In this section we develop a theory of singular integral operators
acting on $\lp{\mu}$ spaces.  

Preliminarly, we observe the following.
Recently, M.~Bownik \cite{Bow}, following
up earlier work of Y.~Meyer, produced in the 
classical Euclidean setting an example of an 
operator $\cT_B$ defined on $(1,\infty)$-atoms with
$$
\sup\{ \norm{\cT_B a}{\lu{\la}}: \hbox{$a$ is a $(1,\infty)$-atom} \} < \infty,
$$
that does not extend to a bounded operator from $\hu{\la}$
to $\lu{\la}$: here $\la$ denotes the Lebesgue measure on $\BR^n$.
The problem of giving
sufficient conditions for an operator uniformly bounded on 
atoms to extend to a bounded operator from $\hu{\la}$ to 
$\lu{\la}$ has been considered independently in \cite{MSV,YZ}.
The paper \cite{YZ} and most of \cite{MSV} focuse
on the Euclidean case.  However, in the last part of \cite{MSV}
more general settings are considered.
In particular, suppose that $(M,\rho,\mu)$ is a $\si$-finite
metric measure space with properties (LDP), (I) and (AMP).
Then the following holds.

\begin{proposition}  \label{p: basic prop}
Suppose that $q$ is in $(1,\infty)$, and that
$\cT$ is a linear operator defined on finite linear
combinations of $(1,q)$-atoms, satisfying
$$
\sup\{ \norm{\cT a}{\lu{\mu}}: \hbox{\emph{$a$ is a $(1,q)$-atom}} \} < \infty.
$$
The following hold:
\begin{enumerate}
\item[\itemno1]
$\cT$ extends to a bounded operator $\wt \cT$ from $\hu{\mu}$
to $\lu{\mu}$ and the transpose operator $\cT^*$ extends
to a bounded operator $(\cT^*)\wt{\phantom a}$ from 
$\ly{\mu}$ to $BMO(\mu)$;
\item[\itemno2]
if $\cT$ is bounded on $\ld{\mu}$, then $\cT$ and $\wt \cT$ are consistent
operators on $\hu{\mu}\cap \ld{\mu}$.
\end{enumerate}
\end{proposition}

\begin{proof}
The result \cite[Thm~4.1 and Prop.~4.2]{MSV}
is stated for spaces of homogeneous type.  
However, the proof extends \emph{verbatim} to our setting.
\end{proof}

Now we assume that $\cT$ is bounded on $\ld{\mu}$ and
that there exists a locally integrable function $k$
off the diagonal in $M\times M$ such that for every 
function $f$ with support of finite measure
$$
\cT f(x) 
= \int_M k(x,y) \, f(y) \wrt \mu(y) 
\quant x \notin \supp f.
$$
We refer to $k$ as to the {\it kernel of} $\cT$.
A straightforward computation shows that the kernel $k^*$
of the (Hilbert space) adjoint $\cT^*$ of $\cT$ is related to
the kernel $k$ of $\cT$ by the formula
\begin{equation} \label{f: kernel adjoint}
k^*(y,x) = \OV{ k(x,y)}.
\end{equation}
In particular, if $\cT$ is self adjoint on $\ld{\mu}$, then
\begin{equation} \label{f: kernel self adjoint}
k(y,x) = \OV{k(x,y)}.
\end{equation}
The next theorem is a version in our case of a classical result
which holds on spaces of homogeneous type. 
\emph{Mutatis mutandis},
its proof is similar to the proof in the classical case.
However, we include a sketch of the proof for the reader's convenience.
See also \cite{MM} for a detailed proof of the
analogous result in the Gaussian case. 

\begin{theorem} \label{t: singular integrals}
Suppose that $b$ is in $\BR^+$ and $b>R_0/(1-\be)$, where
$R_0$ and $\be$ appear in the definition of property (AMP).
Suppose that $\cT$ is a bounded
operator on $\ld{\mu}$ and that its kernel~$k$
is locally integrable off the diagonal of $M\times M$.  
Define $\upsilon_{k}$ and $\nu_{k}$ by
$$
\upsilon_{k}
= \sup_{B\in \cB_{b}} \sup_{x,x'\in B}
\int_{(2B)^c} \mod{k(x,y) - k(x',y)} \wrt \mu (y),
$$
and
$$
\nu_{k}
= \sup_{B\in \cB_{b}} \sup_{y,y'\in B}
\int_{(2B)^c} \mod{k(x,y) - k(x,y')} \wrt {\mu }(x).
$$
The following hold:
\begin{enumerate}
\item[\itemno1]
if $\nu_{k}$ is finite, then $\cT$ extends to a bounded
operator on $\lp{\mu}$ for all $p$ in $(1,2]$ and
from $H^1({\mu})$ to $\lu{\mu}$.
Furthermore, there exists a constant $C$ such that
$$
\opnorm{\cT}{H^1({\mu});\lu{\mu}}
\leq C \, \bigl(\nu_{k} + \opnorm{\cT}{\ld{\mu}}\bigr);
$$
\item[\itemno2]
if $\upsilon_{k}$ is finite, then $\cT$ extends to a bounded
operator on $\lp{\mu}$ for all $p$ in $[2,\infty)$
and from $\ly{\mu}$ to $BMO({\mu})$.
Furthermore, there exists a constant $C$ such that
$$
\opnorm{\cT}{\ly{\mu};BMO({\mu})}
\leq C \, \bigl(\upsilon_{k} + \opnorm{\cT}{\ld{\mu}}\bigr);
$$
\item[\itemno3]
if $\cT$ is self adjoint on $\ld{\mu}$ and $\nu_{k}$ is
finite, then $\cT$ extends to a bounded operator
on $\lp{\mu}$ for all $p$ in $(1,\infty)$, from $H^1({\mu})$ to $\lu{\mu}$
and from $\ly{\mu}$ to $BMO({\mu})$.
\end{enumerate}
\end{theorem}

\begin{proof}
First we prove \rmi.
In view of Proposition~\ref{p: basic prop} it suffices to show that
$\cT$ maps $(1,2)$-atoms in $H_1^1(\mu)$ uniformly into $\lu{\mu}$.
This is done exactly as in the classical case, except that we need to 
consider only atoms supported in balls of $\cB_b$.  
Then $\cT$ maps $\hu{\mu}$ into $\lu{\mu}$, and, by interpolation,
on $\lp{\mu}$ for all $p$ in $(1,2)$. 

Next we prove \rmii. 
Since the kernel $k^*(x,y)$ of the (Hilbert space) adjoint 
$\cT^*$ of $\cT$ is $\OV{{k(y,x)}}$, by \rmi\
$\upsilon_{k^*}=\nu_k$ is finite and the operator 
$\cT^*$ is bounded from $\hu{\mu}$ to $\lu{\mu}$.
Hence $\cT$ is bounded from $L^\infty({\mu})$ to $BMO({\mu})$.  
Moreover  
$$
\opnorm{\cT^*}{\ly{\mu};BMO({\mu})}
\leq C \, \bigl(\nu_{k} + \opnorm{\cT}{\ld{\mu}}\bigr).
$$
By interpolation 
$\cT$ extends to a bounded operator
on $\lp{\mu}$ for all $p$ in $(2,\infty)$,

Finally, we prove \rmiii.  By \rmii, $\cT$ extends to
a bounded operator on $\lp{\mu}$ for all $p$ in $(1,2)$
and from $H^1({\mu})$ to $\lu{\mu}$.
By (\ref{f: kernel self adjoint}) also $\upsilon_{k}$
is finite.  Hence, by \rmi, $\cT$ extends to a bounded
operator on $\lp{\mu}$ for all $p$ in $[2,\infty)$
and from $\ly{\mu}$ to $BMO({\mu})$, thereby
concluding the proof of \rmiii\ and of the theorem.
\end{proof}

\begin{remark} \label{r: singular integrals}
It is worth observing that in the case where $M$
is a Riemannian manifold and the kernel $k$ is ``regular'',
then the condition $\upsilon_{k} <\infty$ of
Theorem~\ref{t: singular integrals}~\rmi\ may be replaced by the condition
$\upsilon_{k}'<\infty$, where
\begin{equation} \label{f: singular integrals}
\upsilon'_{k}
= \sup_{B\in \cB_{b}} \, r_B \sup_{x\in B}
\int_{(2B)^c} \mod{\grad_x k(x,y)} \wrt \mu (y).
\end{equation}
Similarly, the condition $\nu_{k} <\infty$ of
Theorem~\ref{t: singular integrals}~\rmii\ may be replaced by
the condition $\nu_{k}'<\infty$, where
\begin{equation} \label{f: singular integrals II}
\nu'_{m}
= \sup_{B\in \cB_{b}} \, r_B \sup_{y\in B}
\int_{(2B)^c} \mod{\grad_y k(x,y)} \wrt {\mu }(x).
\end{equation}

Indeed, by the mean value theorem we see that the condition
$$
\sup_{B\in \cB_{1}} \sup_{x,x'\in B} \mod{x-x'} \,
\int_{(2B)^c}  \mod{\grad_x k(x,y)} \wrt \mu (y) <\infty
$$
implies the condition $\upsilon_{k} <\infty$ of the theorem.
Since $\mod{x-x'} < 2\,r_B$, (\ref{f: singular integrals}) follows.

We note also that formula (\ref{f: kernel self adjoint})
imply that if $\cT$ is self adjoint, then
$\upsilon'_{k}<\infty$ holds if and only if $\nu'_{k}<\infty$ does.
\end{remark}

\section{Cheeger's isoperimetric constant and property \PP}
	\label{s: Cheeger}

In this section we show that the theory developed in the
previous sections may be applied to an interesting class
of complete noncompact Riemannian manifolds.
First we recall that if 
$(M,\rho)$ is a complete Riemannian manifold with Ricci
curvature bounded from below, then
$M$ is a locally doubling metric space with
respect to the Riemannian measure and the geodesic distance.
The proof of this fact is a direct consequence of 
M.~Gromov's variant \cite{Gr} of R.L.~Bishop's comparison theorem
(see, for instance, \cite{BC}).

It is natural to investigate the dependence of
property \PP\ on other geometric or analytic properties of  $M$.
Denote by $b(M)$ the bottom of the spectrum of $M$, defined by
$$
b(M)
= \inf_{f\neq 0}
\, \frac{\int_M\mod{\grad f}^2\wrt V}{\int_M f^2 \wrt V},
$$
where $V$ denotes the Riemannian measure of $M$
and $f$ runs over all sufficiently smooth functions with compact support. 
We denote by $h(M)$ the Cheeger isoperimetric constant of $M$ defined by
\begin{equation} \label{f: Cheeger}
h(M)=\inf \frac{\si(\partial A)}{V(A)},
\end{equation}
where the infimum runs over all bounded open subsets $A$ 
with smooth boundary $\partial A$ and where $\sigma$ 
denotes the $(d-1)$-dimensional measure. 
Cheeger proved that 
\begin{equation} \label{f: CI}
b(M) \geq h(M)^2/4.
\end{equation}

In this section we shall relate $b(M)$ and $h(M)$ to 
the isoperimetric constant $I_M$ defined in Section~\ref{s: Geom prop}.

First we need to show that in Cheeger's isoperimetric inequality 
$\si(\partial A) \geq h(M) \, V(A)$,
we may replace  
the bounded open set $A$ with smooth boundary by any set $E$ 
of finite measure and the $(d-1)$-dimensional measure 
$\sigma$ of the boundary by 
the perimeter of $E$. The definition of perimeter of a 
set in a Riemannian manifold mimics closely the definition in the Euclidean 
setting \cite{EG,MPPP}. 

If $U$ is an open subset of $M$ 
we shall denote by $\Lambda^k_c(U)$ the 
space of smooth $k$-forms  with compact support contained in 
$U$. 
The divergence is the formal adjoint of the the exterior 
derivative $\wrt$, i.e. the operator $\Div$
mapping $k+1$-forms to $k$-forms defined by
\begin{equation}\label{div}
\int_M \langle \Div \om,\eta\rangle_x \wrt V(x)
=-\int_M \langle \om,\wrt\eta\rangle_x \wrt V(x)
\end{equation}
for all smooth $k+1$-forms $\om$ and all 
smooth $k$-forms $\eta$ with compact support.

Given a real valued function $f$ in $L^1(M)$, the variation of $f$ in $U$ is
$$
\Var(f,U)
= \sup\set{\int_M f\, (\Div \om) \wrt V: 
\om\in\Lambda^1_c(U), \norm{\om}{\infty}\le 1}.
$$
We say that  $f$ has bounded variation in $U$ 
if $\Var(f,U)<\infty$. 
We shall denote by $BV(U)$ the space of function of bounded variation in $U$. 

As in the Euclidean case, if $f$ is in $BV(M)$, then the map 
$U\mapsto \Var(f,U)$ extends to a finite Borel measure on $M$. 

A measurable set $E\subset M$ has finite perimeter 
if its indicator function $\1_E$ is in $BV(M)$. 
If $U$ is a Borel set the perimeter of $E$ in $U$ is 
$$
P(E,U)=\Var(\1_E,U).
$$
Since the manifold $M$ is complete, by \cite[Theorem 2.7]{Hebey}  
the Sobolev space $H^{1,1}(M)$ is the completion of the space 
$C^\infty_c(M)$ of smooth functions 
with compact support on $M$ with respect to the norm 
$$
\norm{f}{H^{1,1}}
=\norm{f}{1}+\int_M \mod{\grad f} \wrt V.
$$
It is an easy matter to show that $H^{1,1}(M)\subset BV(M)$ and 
$$
\Var(f,M)
= \int_M \mod{\grad f} \wrt V\quant f\in H^{1,1}(M).
$$
Note also that the space $\Lip_c(M)$ of Lipschitz functions 
with compact support on $M$ is contained in $H^{1,1}(M)$, 
by \cite[Lemma 2.5]{Hebey}.

In the Euclidean setting it is well known that $BV$ 
functions may be approximated in variation by smooth functions in $L^1$ 
(see, for instance \cite[Theor. 3.9]{EG}). 
The same result holds in the Riemannian setting \cite[Prop. 1.4]{MPPP}.

\begin{lemma}\label{l: approximation}
For every $f$ in $BV(M)$ there 
exists a sequence $(f_n)$ in $C^\infty_c(M)$ 
which converges to $f$ in $L^1(M)$ and such that 
$$
\Var(f,M)=\lim_{n\to\infty} \Var(f_n,M).
$$
\end{lemma}
Now we can prove that Cheeger's isoperimetric inequality for 
smooth compact hypersurfaces implies an analogous inequality for 
the perimeters of sets of finite measure.

\begin{proposition}\label{p: cheeger for bd}
Suppose that $M$ is a complete Riemannian manifold. If $h(M)>0$ 
then for every measurable set $E$ 
$$
P(E,M)\ge h(M)\,  V(E).
$$
\end{proposition}

\begin{proof}
It is well known that Cheeger's isoperimetric 
inequality for smooth submanifolds is equivalent to the Sobolev inequality
\begin{equation}\label{Sob}
\Var(f,M)=\int_M \mod{\grad f}\wrt V\ge h(M) \int_M\mod{f}\wrt V
\end{equation}
for all real valued functions in $C^1_c(M)$ \cite{Chavel}.  
Suppose that $E$ is a measurable set of finite perimeter. 
By Lemma~\ref{l: approximation}  there exists a sequence $(f_n)$ 
of functions in $C^1_c(M)$ such that $f_n\to \1_E$ in 
$L^1(M)$ and $\Var(f_n,M)\to\Var(\1_E,M)=P(E,M)$. 
Hence the desired conclusion follows from (\ref{Sob}).
\end{proof}

The following lemma is the coarea formula for  functions in $H^{1,1}(M)$.
The proof uses the density of $C^1_c(M)$ in $H^{1,1}(M)$ 
and mimics closely the argument in the Euclidean setting \cite{EG}.

\begin{lemma}\label{l: coarea} 
Suppose that $M$ is a complete Riemannian manifold.  
Assume that $f\in H^{1,1}(M)$. 
Then for every open subset $U$ of $M$ the sets 
$A(t) : =\set{x\in U: f(x)>t}$ have finite perimeter for a.e. $t$ in $\BR$ and
$$
\int_U\mod{\grad f} \wrt V=\int_{\BR} P\big( A(t),U\big)\wrt t.
$$
\end{lemma} 

\begin{remark}
We observe {\it en passant} that using Lemma~\ref{l: approximation}, one can prove the following coarea formula for functions in $BV(M)$
$$
\Var(f,M)=\int_{\BR} P\big(A(t),M\big)\wrt t.
$$
\end{remark}
Now we are ready to prove the equivalence 
of property \PP\ and Cheeger's inequality. 
We recall that the constant $I_M$ is defined in Section~\ref{s: Geom
prop}.

\begin{theorem} \label{t: Cheeger and (P)} 
Suppose that $M$ is a complete Riemannian manifold.  The following hold:
\begin{enumerate}
\item[\itemno1]
$M$ possesses property \PP\ if and only if $h(M)>0$.
Furthermore $I_M = h(M)$;
\item[\itemno2]
if the Ricci curvature of $M$ is bounded from below,
then $M$ possesses property \PP\ if and only if $b(M)>0$.
\end{enumerate}
\end{theorem}

\begin{proof}
First we prove \rmi.
Assume that $M$ possesses property \PP\ and denote by $A$ 
a bounded open subset of $M$ with smooth  boundary. 
By Proposition \ref{p: prop PP} 
$$
\frac{V(A_\kappa)}{\kappa}
\ge \,\frac{1-\e^{-I_Mt}}{t}\, V(A) \quant t\in\BR^+.
$$
Hence
$$
\liminf_{\kappa\to 0} \frac{V(A_\kappa)}{\kappa} 
\ge I_M \,V(A).
$$ 
Since the limit in the left hand side 
is the lower inner Minkowski content of $\partial A$, 
which coincides with $\sigma(\partial A)$ because $\partial A$ is smooth, we have proved that 
Cheeger's isoperimetric constant $h(M)$ is at least $I_M$.

To prove the converse, assume that $h(M)>0$ 
and let $A$ be a bounded open set in $M$. 
Since the manifold $M$ is complete, the function $f$ defined by 
$f(x)=\rho(x,A^c)$ (here $A^c$ denotes the complementary set of $A$ in $M$)
is Lipschitz and $\mod{\grad f}=1$ almost everywhere. 
Recall that for each $t$ in $\BR$ we denote by $A^t$ the set 
$\set{x\in A: f(x)>t}$.  Notice that 
$$
P\big(A^s,A^t\big)
= \begin{cases}
P \big(A^s,M\big) & {\text{ if}}\  s>t;\\
0                  & {\text{ if}}\  s<t. \\
\end{cases}
$$
Thus, by the coarea formula 
$$
V\big(A^t\big)=\int_{A^t}\mod{\grad f}\wrt V 
=\int_t^\infty P\big(A^s,M\big) \wrt s.
$$
Hence, by  Proposition~\ref{p: cheeger for bd}
$$
\frac{\wrt}{\wrt t}V\big(A^t\big)
= -P\big(A^t,M\big)\le-h(M)\,V\big (A^t\big) \qquad \text{for a.e.}\ t\in\BR.
$$
This differential inequality implies that 
$V\big(A^t\big)\le e^{-h(M)\,t} V(A)$ for all $t>0$, i.e.
\begin{align*}
V(A_t)
& =    V(A)-V\big(A^t\big) \\ 
& \ge  (1-e^{-h(M)\,t})\,V(A) \\ 
& \ge  h(M) \,t\, V(A) \quant t\in (0,1).
\end{align*}
Thus $M$ possesses property \PP, and $h(M) \leq I_M$, as required
to conclude the proof of \rmi.

To prove \rmii\ we recall that
if $M$ has Ricci curvature bounded below by $-K$, for some $K\ge0$ then
\begin{equation}\label{Buser}
b(M)\le C\big(\sqrt{K}\ h(M)+h(M)^2\big),
\end{equation}
where $C$ is a constant which 
depends only on the dimension of $M$ \cite{Bu, Le}.
This inequality, together with Cheeger's inequality (\ref{f: CI}),
shows that the constants $h(M)$ and 
$b(M)$ are equivalent.  The required conclusion follows directly
from \rmi. 
\end{proof}

\begin{remark}
We remark that property \PP\ is invariant under
quasi-isometries.   Indeed, the fact that $h(M)$
is positive is invariant under quasi-isometries \cite[Remark~V.2.2]{Chavel}. 
\end{remark}

\begin{remark}
Suppose that $M$ is a complete Riemannian manifold. 
If $b(M) = 0$, then $M$ has not property \PP\ by Theorem~\ref{t: Cheeger
and (P)}~\rmi.

Observe also that if $M$ has Ricci curvature bounded from below,
and a spectral gap, then~$M$ has property \PP.
In particular noncompact Riemannian symmetric spaces 
and Damek-Ricci spaces have property \PP.
\end{remark}

\section{Applications} \label{s: Applications}

In this section we illustrate some applications
of the theory developed in the previous sections.
Other applications will appear in \cite{CMM}. 

The first application we consider is to spectral multipliers
on certain Riemannian manifolds.   
Suppose that $M$ is a complete Riemannian manifold with
positive injectivity radius $\inj(M)$, Ricci curvature
bounded from below, and positive Cheeger isoperimetric constant $h(M)$.  
By Cheeger's inequality the bottom $b(M)$ of the $\ld{M}$ spectrum
of the Laplace--Beltrami operator $\cL$ on $M$ is positive.
We denote by $\mu$ the Riemannian measure on $M$.

As shown in Section~\ref{s: Cheeger}, under these assumptions
$M$ possesses property \PP.  Furthermore,
there exist constants $\al$, $\al'$,
$\be$, $\be'$, $C_1$ and $C_2$ such that
$$
C_1 \,(1+r^2)^{\al'/2} \, \e^{\be'\, r}
\leq \mu\bigl(B(p,r)\bigr)
\leq C_2 \, (1+r^2)^{\al/2} \, \e^{\be\, r}
\quant r \in \BR^+.
$$
We say that $M$ has ($N$-)bounded geometry provided that the derivatives of
the Riemann tensor up to the order $N$ are uniformly bounded.
Clearly if $M$ has $N$-bounded geometry, then the Ricci
curvature of $M$ is bounded from below. 

For each $\si$ in $\BR^+$ we denote by $\bS_{\si}$ the
strip $\{\zeta \in \BC: \mod{\Im \zeta} < \si \}$.

\begin{definition}  \label{d: HinftySbetakappa}
Suppose that $\kappa$ is a positive integer and that $\si$
is in $\BR^+$.
The space $H^\infty (\bS_{\si};\kappa)$ is the vector space of 
all functions $f$ in $H^\infty (\bS_{\si})$ for which
there exists a positive constant $C$ such that
for each $\vep$ in $\{-1,1\}$
\begin{equation} \label{f: Ssigmakappa}
\mod{D^k f(\zeta)} \leq {C} \, (1+\mod{\zeta})^{-k}
\quant k \in \{0,1,\ldots,\kappa\} \quant \zeta\in \bS_{\si}.
\end{equation}
If (\ref{f: Ssigmakappa}) holds, 
we say that $f$ satisfies a \emph{Mihlin--H\"ormander condition 
of order $\kappa$} on the strip~$\bS_{\si}$.  We endow 
$H^\infty (\bS_{\si};\kappa)$ with the norm 
$$
\norm{f}{\si;\kappa}
= \max_{k \in \{0,1,\ldots,\kappa\}}  \sup_{\zeta\in \bS_{\si}}
\, (1+\mod{\zeta})^{k} \, \mod{D^k f(\zeta)}.
$$
\end{definition}

The following result complements a celebrated result of Taylor
\cite[Thm~1]{Ta}.

\begin{theorem}
Suppose that $M$ is an $n$ dimensional complete Riemannian manifold with
$N$-bounded geometry, where $N$ is an integer $\geq n/2+1$.
Assume that the injectivity radius $\inj(M)$ and the bottom $b(M)$
of the $\ld{M}$ spectrum of $M$ are positive.  
Suppose that $f$ is in $H^\infty (\bS_{\si};\kappa)$, where $\si
\geq \be/2$ and $\kappa > \max(\al/2+1, n/2+1)$.   Then $f(\cL_b^{1/2})$
extends to a bounded operator from $\hu{\mu}$, from $\ly{M}$ to $BMO(\mu)$
and on $\lp{M}$ for all $p$ in $(1,\infty)$.
\end{theorem}

\begin{proof}
Denote by $\cL_b$ the operator $\cL -b\, \cI$, formally
defined on $\ld{\mu}$. 
The strategy of the proof of \cite[Thm~1]{Ta} is to decompose
the operator $f(\cL_b^{1/2})$ as the sum of two operators,
$f_0(\cL_b^{1/2})$ and $f_\infty(\cL_b^{1/2})$,
where $f_\infty(\cL_b^{1/2})$ is bounded on 
$\lu{\mu}$ and $\ly{\mu}$ and $f_0(\cL_b^{1/2})$ is of weak type $1$.

To prove the latter, Taylor \cite{Ta} and 
Cheeger, Gromov and Taylor \cite{CGT} prove that the integral kernel
of $f_0(\cL_b^{1/2})$, which is compactly supported,
satisfies a H\"ormander type integral condition.
Taking this for granted, by Theorem~\ref{t: singular integrals},
the operator $f_0(\cL_b^{1/2})$ is bounded from
$\hu{\mu}$ to $\lu{\mu}$, and from $\ly{\mu}$ to $BMO(\mu)$. 
Therefore the same is true for $f(\cL_b^{1/2})$. 
The boundedness of $f(\cL_b^{1/2})$ on $\lp{\mu}$ the follows
by interpolation.
\end{proof}

Note that this result applies to Riemannian symmetric spaces
of the noncompact type, and to Damek--Ricci spaces. 
In the case where $M$ is a symmetric space of the 
noncompact type and real rank $>1$, J.Ph.~Anker \cite{A}
extended Taylor's result \cite[Thm~1]{Ta} to 
certain multiplier operators for the spherical Fourier transform.  

Suppose that $G$ is a noncompact semisimple Lie group
with finite centre, and denote by~$K$ a maximal 
compact subgroup thereof, and by $X$ the associated 
Riemannian symmetric space of the noncompact type $G/K$
(the $G$ invariant metric on $X$ is induced by the Killing form of $G$).
Denote by $\mu$ a $G$-invariant measure on $X$.
We complement Anker's result by showing
that (some) of the multiplier operators
he considers satisfy natural $\hu{\mu}$-$\lu{\mu}$
and $\ly{\mu}$-$BMO(\mu)$ estimates.  

Suppose that $G = KAN$ is an Iwasawa decomposition of
$G$, where $A$ is abelian and $N$ is nilpotent.  
Denote by $\frg$ and $\fra$ the Lie algebras of
$G$ and $A$ respectively, and by $\rho$ the half sum
of the positive roots of $(\frg,\fra)$ with
multiplicities.  Denote by $\fra^*$ the dual of $\fra$.
An important r\^ole in what follows 
is played by a certain tube $\bT$ in the complexified dual
$\fra_{\BC}^*$ of~$\fra$.  To define~$\bT$, denote by $\bW$ the
convex hull of the vectors $\{w\cdot \rho:
w \in W\}$ in~$\fra^*$, where $W$ denotes the Weyl group of $G$.
Then define $\bT = \fra^*+i\bW$. 

\begin{definition}   \label{d: HinftySbetakappa Symm} 
Suppose that $\kappa$ is a positive integer.
The space $H^\infty (\bT;\kappa)$ is the vector space of 
all Weyl invariant bounded holomorphic functions $m$ on $\bT$
for which there exists a positive constant $C$ such that
for every multiindex $\be$ with $\mod{\be}\leq \kappa$
\begin{equation} \label{f: Ssigmakappa Symm}
\mod{D^{\be} m(\zeta)} \leq {C} \, (1+\mod{\zeta})^{-\mod{\be}}
\quant \zeta\in \bT.
\end{equation}
If (\ref{f: Ssigmakappa Symm}) holds, 
we say that $f$ satisfies a \emph{Mihlin--H\"ormander condition 
of order $\kappa$} on the tube~$\bT$.  We endow 
$H^\infty (\bT;\kappa)$ with the norm 
$$
\norm{f}{\kappa}
= \max_{\mod{\be} \leq \kappa}  \sup_{\zeta\in \bT}
\, (1+\mod{\zeta})^{\mod{\be}} \, \mod{D^\be f(\zeta)}.
$$
\end{definition}

Suppose that $k$ is a $K$-bi-invariant distribution
on $G$, and denote by $m$ its spherical Fourier transform:
$m$ may be thought of as a distribution on $\fra^*$.
If $m$ is bounded on $\fra^*$, then the convolution operator 
$\cT_k$, defined by
$$
\cT_k f = f * k
\quant f \in C_c^\infty (G)
$$
extends to a bounded operator on $\ld{\mu}$.  

\begin{theorem}
Suppose that $\kappa>(\dim X)/2+1$.  Suppose that
$k$ is a $K$-bi-invariant distribution such that
its spherical transform $m$ is in $H^{\infty}(\bT; \kappa)$.
Then the operator $\cT_k$ extends to a bounded operator 
from $\hu{\mu}$ to $\lu{\mu}$ and from $\ly{\mu}$ to $BMO(\mu)$.
\end{theorem}

\begin{proof}
Denote by $\psi$ a $K$-bi-invariant smooth function
on $G$ with compact support which is identically
$1$ in a neighbourhood of the identity, and define the 
distributions $k_0$ and $k_\infty$ by
$$
k_0 = \psi \, k
\qquad\hbox{and}\qquad
k_\infty = (1- \psi) \, k.
$$
Anker \cite[Thm~1]{A} shows that
$k_\infty$ is, in fact, a function in $\lu{G}$.
Therefore, the operator $\cT_{k_\infty}$, defined by
$\cT_{k_\infty} f= f\mapsto f*k_\infty$
extends to a bounded operator on $\lu{\mu}$ and on $\ly{\mu}$, 
and \emph{a fortiori}
to a bounded operator from $\hu{\mu}$ to $\lu{\mu}$ and 
from $\ly{\mu}$ to $BMO(\mu)$.  

Furthermore, Anker proves that
$k_0$ is locally integrable off the origin, and
satisfies the following H\"ormander type integral~inequality 
$$
\sup_{B\in \cB_1} \sup_{y \in B} \int_{(2B)^c}
\mod{k_0(y^{-1}x) - k_0(x)} \wrt \mu(x)
< \infty.
$$
Define the operator $\cT_{k_0}$ by $\cT_{k_0}f = f * k_0$.
Denote by $\Delta$ the diagonal in $X\times X$,
and define the locally integrable function 
$t$ on $X\times X \setminus \Delta$ by
$$
t(x,y) = k_0(y^{-1}x).
$$
It is straightforward to check that $t$ is the kernel (see
the definition  at the beginning of Section~\ref{s: Singular integrals})
of the operator $\cT_{k_0}$, and that $t$ satisfies
conditions $\nu_t<\infty$ and $\upsilon_t < \infty$.
By Theorem~\ref{t: singular integrals}~\rmiii\  the operator $\cT_{k_0}$
extends to a bounded operator from $\hu{\mu}$ to $\lu{\mu}$ and
from $\ly{\mu}$ to $BMO(\mu)$.

Since $\cT_k = \cT_{k_0} + \cT_{k_\infty}$,
the required boundedness properties of $\cT_k$ follow
directly from those of $\cT_{k_0}$ and $\cT_{k_\infty}$.
\end{proof}

Our last application is to the boundedness of Riesz transforms.
This is a very fashionable and interesting subject: see
\cite{CD,ACDH} for recent results on manifolds, and the references
therein for less recent results. 
Suppose that $M$ is a complete Riemannian manifolds satisfying
the following assumptions:  the Riemannian measure $\mu$ is 
locally doubling, and the following
scaled local Poincar\'e inequality holds:
for every positive $b$ there exists a constant $C$ such that
for every $B$ in $\cB_b$ and for every 
$f$ in $C^{\infty}\bigl(2B\bigr)$
$$
\int_B\mod{f-f_B}^2 \wrt \mu
\leq C \, r^2 \int_{2B} \mod{\nabla f}^2 \wrt \mu.
$$
Suppose also that the volume growth of $M$ is at most exponential.
Note that these assumptions hold if $M$ 
is a Riemannian manifold with Ricci curvature bounded
from below. 

We may define the ``localised'' Riesz transforms $\nabla (\cL + \vep)^{-1/2}$,
where $\vep $ is in $\BR^+$.  
Russ \cite{Ru} proved that the localised Riesz transforms
map local atoms uniformly into $\lu{M}$.  However,
in general, there is no indication that this result
interpolates with the trivial $\ld{M}$ estimate 
to produce $\lp{M}$ boundedness for $p$ in $(1,2)$.

Our theory complements Russ' results.  Indeed, 
Proposition~\ref{p: basic prop} implies
that $\nabla (\cL + \vep)^{-1/2}$ extends to a
bounded operator from $\hu{M}$ into $\lu{M}$.
In the case where $M$ possesses property \PP\ these result interpolate
with the trivial $\ld{M}$ estimate and give that
$\nabla (\cL + \vep)^{-1/2}$ extends to a bounded operator on 
$\lp{M}$ for all $p$ in $(1,2)$, a fact already known,
but whose proof is far from being trivial.

\end{document}